\documentclass[a4paper,10pt,reqno]{amsart}
\usepackage{color}
\usepackage{enumerate}
\usepackage{amssymb,amsmath,amsthm,latexsym,mathrsfs}
\usepackage{amsbsy,amsfonts,mathtools,slashed}
\usepackage{bm}
\usepackage{graphicx,color,yfonts}
\usepackage[numbers,sort]{natbib}
\usepackage[latin9]{inputenc}
\usepackage[pagewise]{lineno}

\usepackage[colorlinks=true]{hyperref}
\hypersetup{linkcolor=red,citecolor=blue,filecolor=dullmagenta,urlcolor=blue}

\newtheorem{thm}{Theorem}[section]
\newtheorem{lemma}[thm]{Lemma}
\newtheorem{prop}[thm]{Proposition}
\newtheorem{cor}[thm]{Corollary}
\newtheorem{defn}[thm]{Definition}
\newtheorem{rem}[thm]{Remark}

\makeatother
\numberwithin{equation}{section}
\numberwithin{figure}{section}
\theoremstyle{plain}
\theoremstyle{plain}
\theoremstyle{plain}

\newcommand{\les}{\lesssim}

\newcommand{\ve}{{\varepsilon}}
\newcommand{\de}{{\delta}}

\newcommand{\al}{{\alpha}}

\newcommand{\R}{{\mathbb R}}

\def\cF{\mathcal F}

\def\cH{\mathcal H}

\def\cR{\mathcal R}
\def\cS{\mathcal S}

\newcommand{\p}{\partial}

\def\normo#1{\left\|#1\right\|}

\def\bra#1{\left\langle #1\right\rangle}

\def\wt#1{\widetilde{#1}}
\def\wh#1{\widehat{#1}}
\def\ol#1{\overline{#1}}

\newcommand{\palbe}{{P_{\al,\beta}^\pm}}

\def\sgn#1{{\rm sgn}(#1)}

\begin{document}
\title{Global solution and asymptotic behavior for the kinetic derivative NLS on $\mathbb R$}

\author{Nobu Kishimoto}
\address{Research Institute for Mathematical Sciences(RIMS), Kyoto University, Kitashirakawa Oiwake-cho, Sakyo-ku, Kyoto, 606-8502, Japan}
\email{nobu@kurims.kyoto-u.ac.jp}

\author{Kiyeon Lee}
\address{Department of Mathematics, University of Michigan, Ann Arbor, MI, USA}
\email{kiyeonl@umich.edu}

\thanks{{\bf 2020 Mathematics Subject Classification:} Primary -- 35Q55; Secondary -- 35A01, 35B40, 35B45.}
\thanks{{\it Keywords and phrases.} kinetic derivative NLS, global existence, asymptotic behavior, modified scattering, a priori estimate.}

\begin{abstract}
In this paper we investigate the global well-posedness and long-term behavior of solutions to the kinetic derivative nonlinear Schr\"odinger equation (KDNLS) on the real line. The equation incorporates both local cubic nonlinearities with derivative terms and a non-local term arising from the Hilbert transform, modeling interactions in plasma physics. We establish global existence for small initial data in the weighted Sobolev space $H^2 \cap H^{1,1}$  and optimal time decay effect. Using energy methods and a frequency-localized gauge transformation, we overcome the difficulties posed by the non-local nonlinearities and provide a rigorous analysis of the asymptotic behavior. Our results also describe modified scattering phenomena with a suitable phase modification, showing that the solutions exhibit a precise asymptotic profile as $t \to \infty$.
\end{abstract}

\maketitle

\tableofcontents

\section{Introduction}
\label{sec:intr}

We consider the following one-dimensional nonlinear Schr\"odinger equation (NLS) with cubic nonlinear terms with derivative:
\begin{equation}
	\left\{\begin{aligned}
		i\partial_tu + \partial_x^2u &= i\alpha \partial_x \big[ |u|^2u\big] +i\beta \partial_x\big[\mathcal{H}(|u|^2)u\big] , \qquad t>0, ~ x \in \mathbb{R},\\
		u(0,x) &= \phi(x),
	\end{aligned}\right.\label{kdnls}
\end{equation}
where $\alpha ,\beta \in \mathbb{R}$ are constants and $\cH$ is the Hilbert transform defined by
\[ \mathcal{H}f(x) :=\frac{1}{\pi} p.v.\int_{-\infty}^\infty\frac{f(y)}{x-y}\,dy = \mathcal{F}^{-1}\left[ -i\, \mathrm{sgn}(\xi)\mathcal{F}f\right] (x). \]
When $\beta =0$, the equation \eqref{kdnls} reduces to the standard derivative NLS, which is known to be completely integrable.
The equation \eqref{kdnls} with $\beta \neq 0$ was derived in \cite{MW86,MW88} as a model for the propagation of weakly nonlinear and weakly dispersive Alfv\'en waves in a collision-less plasma.
The non-local nonlinear term $\p_x\left[ \cH(|u|^2)u\right]$ representing the effect of resonant particles is obtained using a kinetic model, and thus we call \eqref{kdnls} with $\beta \neq 0$ the \emph{kinetic derivative NLS}.

We begin by observing a fundamental property of solutions to \eqref{kdnls}.
Formally, the  $L^2$-norm satisfies the identity
\begin{equation*}
	\frac{d}{dt}\| u(t)\|_{L^2}^2\ =\ \beta \big\| D_x^{\frac12}(|u(t)|^2)\big\|_{L^2}^2,
\end{equation*}
where $D_x:=|\partial_x|=\mathcal{H}\partial_x$. This identity implies the dissipative nature of the equation when $\beta <0$, in contrast to the monotone increasing nature when $\beta >0$.

This study extends the local and global well-posedness results for \eqref{kdnls} with $\beta \neq 0$ obtained in \cite{KL-LWP}. As the asymptotic behavior and time decay effect of solutions were not addressed in \cite{KL-LWP}, we aim to investigate these properties in this paper. We begin by reviewing the known results for \eqref{kdnls} and related equations. While most of these results have been discussed in detail in \cite{KL-LWP}, we shall recall them here for the convenience of the reader.

The mathematical theory for \eqref{kdnls} has been extensively developed in the special case $\beta=0$, corresponding to the derivative NLS.
Global well-posedness in the scale-critical space $L^2(\R)$ has been established by Harrop-Griffiths, Killip, Ntekoume, and Visan \cite{HGKNVp} thanks to complete integrability (see also \cite{HGKV23} for the torus case). For the case $\beta \neq 0$, Rial~\cite{R02} showed the construction of  global $L^2$-weak solutions to \eqref{kdnls} with $\beta <0$ and smoothing effect by the dissipative structure. We also refer to Peres~de~Moura and Pastor~\cite{PdMP11} for the applicability of Kato smoothing and maximal function estimates to the non-local nonlinearity. The first author and Tsutsumi~\cite{KT-stFr,KT23-2} established local well-posedness in $H^s(\R)$ for $s>1$ and derived a priori $H^s$ bounds for solutions with $s>1/4$, again under the dissipative assumption $\beta <0$. Furthermore, in the periodic setting \cite{KT23-1}, they observed that the dissipative structure makes the equation of (first-order) parabolic type by the resonant interactions. This observation led to global well-posedness in $H^s(\mathbb{T})$, $s>1/4$ in the case $\beta<0$~\cite{KT-gauge} and non-existence of solutions for $s>3/2$ in the case $\beta >0$~\cite{KT23-1}. As discussed above, the authors of this paper have also recently proved well-posedness in $H^{2}(\R) \cap H^{1,1}(\R)$ locally for $\beta \neq 0$ and globally $\beta <0$.

As reviewed above,  the asymptotic behavior of the solution to \eqref{kdnls} with $\beta \neq 0$ on the real line has not been shown. In this paper, we prove the modified scattering and decay properties for \eqref{kdnls} with $\beta \neq 0$. Moreover, we examine the precise dissipative effect of the solution on $\R$ thanks to the decay effects. Here is the main theorem of this paper:
\begin{thm}\label{thm:main}
	There exists $\ve_0>0$ such that the following holds:  Suppose that the initial data $\phi$ in \eqref{kdnls} satisfies 
\begin{align}\label{thm:initial}
	\|\phi\|_{H^2} + \|x\phi\|_{H^1}=: \ve \le \ve_0,
\end{align}
for some $0 < \ve \le \ve_0$. Then the Cauchy problem \eqref{kdnls} with initial data $\phi$ has a global solution $u\in C([0,\infty),H^2\cap H^{1,1})$ to \eqref{kdnls}, decaying as
\begin{align}\label{thm:decay}
	\|u(t)\|_{W^{1,\infty}} \les  \bra{t}^{-\frac12}\ve
\end{align}
and there exist functions $W, \,\Phi \in L^\infty$ and $\de_1 >0$ such that
\begin{align}\label{thm:asymptotic}
\begin{aligned}
		u(t)&= \frac1{\sqrt{2it}} e^{\frac{ix^2}{4t}} W\left(\frac x{2t}\right) \exp\left(i\left[ \frac {\al x}{4t}\left|W\left(\frac x{2t}\right) \right|^2  +  \frac {\beta x}{4t} \cH \left(\left|W\right|^2\right) \left(\frac x{2t}\right) \right] \log t + i \Phi\left(\frac x{2t}\right)\right)  \\
	&\qquad  +O(t^{-\frac12-\de_1})
\end{aligned}
\end{align}
uniformly in $x \in \R$. Moreover, the limit $D_\infty =  \lim\limits_{t\to \infty}  \|u(t)\|_{L^2}$ exists and satisfies $D_\infty = (1 + O(\ve^2)) \|\phi\|_{L^2}$ and for some $\de_2>0$,
\begin{align}\label{thm:dissipation}
\begin{aligned}
	&D_\infty \le \|\phi\|_{L^2}, \qquad	   0 \le \|u(t)\|_{L^2} -D_\infty \les  \bra{t}^{-1+\de_2}\ve^3,    \qquad\mbox{ when }\;\; \beta<0,\\
	&D_\infty \ge \|\phi\|_{L^2}, \qquad 0 \le D_\infty - \|u(t)\|_{L^2} \les \bra{t}^{-1+\de_2}\ve^3 ,  \qquad \mbox{ when }\;\; \beta>0,
\end{aligned}
\end{align}
for all $t>0$.
\end{thm}

Note that the equation lacks exact time reversal symmetry when $\beta \neq 0$; if $u(t,x)$ is a solution, then $\ol{u(-t,-x)}$ solves the same equation but with $-\beta$ replacing $\beta$. By this transformation, we have similar behavior results for the backward time.


\begin{rem}
For the asymptotic expansion \eqref{thm:asymptotic}, we prove that there exists $W \in L_\xi^\infty$ such that
\begin{align*}
		\normo{ \bra{\xi}  \left( \wh{u}(t,\xi) - e^{iB(t,\xi)}e^{it\xi^2}   W(\xi) \right) }_{L_\xi^\infty} \les \bra{t}^{-\de} \ve^3,
\end{align*}
for some $\delta >0$, where the phase modification $B(t,\xi)$ is the real-valued function defined by 
\begin{align*}
	B(t,\xi) &:= B_1(t,\xi) +  B_2(t,\xi),\\
	B_1(t,\xi)&:= \al \int_1^t \frac1{2s} \xi  \left|\wh{f}(s,\xi)\right|^2 ds,	\;\;\mbox{ and }\;\;	B_2(t,\xi):= \beta \int_1^t \frac1{2s} \xi \mathcal{H}\Big( \left|\wh{f}(s,\cdot )\right|^2\Big) (\xi) ds,
\end{align*}
with $f(t)= e^{-it\p_x^2} u(t)$. This statement directly implies \eqref{thm:asymptotic} with a suitable $\delta_1\in (0,\delta)$ by using the same argument as in \cite{hayashi-naumkin1998,HN1998-dnls} for an appropriate real-valued function $\Phi$ (see Section \ref{sec:main}).
\end{rem}

\begin{rem}
	When $\beta<0$, Theorem \ref{thm:main} shows that the $L^2$ norm of the solution does not converge to zero as $t\to\infty$ if the   $L^2$ norm of the initial data is not identically zero. This behavior is different from that of other dissipative nonlinear Schr\"odinger equations such as
	\[
	i\partial_t u+\partial_x^2u=-i|u|^2u,
	\qquad x\in\mathbb{R},\quad t>0,
	\]
	for which the large-time behavior was studied by Shimomura \cite{Sh2006} and Kita--Shimomura \cite{KS2007,KS2009}. In particular they showed that 
	\[
	 \| u(t) \|_{L^\infty} \lesssim (t \log t)^{-1/2}, \qquad \| u(t) \|_{L^2} \to 0 \quad (t\to \infty).
	 \]
	  Thus, although the kinetic nonlinearity produces dissipative or accretive effects depending on the sign of $\beta$, the asymptotic behavior in Theorem~\ref{thm:main} is qualitatively different from the damping NLS case.
\end{rem}

\begin{rem}
We establish the monotonic decreasing (resp., increasing) behavior \eqref{thm:dissipation} of the solution to \eqref{kdnls} with $\beta < 0$ (resp., $\beta>0$) thanks to the weighted assumption and null-form bilinear estimates (Lemma \ref{lem:esti-null}). See Lemma \ref{lem:L2energy} and Section \ref{sec:main} for the detail. 

Theorem \ref{thm:main} also establishes the global existence of solutions for \eqref{kdnls} when $\beta > 0$, contrasting with previous results by the first author and Tsutsumi, which demonstrated the non-existence of solutions in the periodic case.
\end{rem}


As we observed in \eqref{thm:asymptotic}, we consider the modified scattering for \eqref{kdnls} with $\beta \neq 0$.  The study of modified scattering for nonlinear dispersive equations has a long history, originating from the works on the NLS. Pioneering results by Ozawa \cite{ozawa1991}, Hayashi-Ozawa \cite{HO1994}, and Hayashi-Naumkin \cite{hayashi-naumkin1998, HN1998-dnls} established the modified scattering behavior for small solutions to cubic NLS and derivative NLS.

Subsequent developments have introduced a convenient framework that explains all of these results in terms of the concept of \emph{space-time resonances}, leading to a deeper understanding of asymptotic behaviors for a wide class of dispersive equations. These approaches were introduced independently by Germain-Masmoudi-Shatah \cite{gemasha2008, gemasha2012-annals, gemasha2012-jmpa} and Gustafson-Nakanishi-Tsai \cite{gunakatsa2009} and developed by many mathematicians \cite{pusa,iopu2014, iopu2015, iopau2019,CKLY2022,sautwang2021-compde,iopau2014,chle2024,hanipusha2013} for the seminal modified (or linear) scattering results on dispersive equations and related systems. Moreover, we refer to \cite{kapu} for an application of this approach based on phase correction
to the modified scattering problems for the NLS studied in \cite{ozawa1991,hayashi-naumkin1998}; see also \cite{ifrtata2015}.

Our goal of this paper is to establish the global existence and asymptotic behavior of small solutions to \eqref{kdnls} with $\beta \neq 0$. The strategy relies on the resonance approach listed above with a bootstrap argument under a suitable assumption, which involves $H^2$ and weighted norms:
\begin{align}\label{intro:assumption}
	\sup_{t \in [0,T]} \left( \langle t \rangle^{-\delta} \|u(t)\|_{H^2} + \langle t \rangle^{-2\delta} \|x e^{-it\partial_x^2} u(t)\|_{H^1} \right) \leq \varepsilon_1
\end{align}
for some $T > 0$, sufficiently small $\varepsilon_1 > 0$, and $\delta > 0$. Under this assumption, we obtain the optimal time decay effect given by \eqref{thm:decay} by employing the $L^\infty_\xi$ norm (see Proposition \ref{prop:linfty}).

To obtain the asymptotic behavior, we need to employ a function space that avoids time growth, unlike in the bootstrap assumption \eqref{intro:assumption}, following similar approaches used for the NLS and the derivative NLS (DNLS) \cite{kapu,HN1998-dnls,HN1997-dnls,HO1994}. Specifically, we extract the resonance term based on the space-time resonant set: For  $f(t):=e^{-it\p_x^2}u(t)$, we have
\begin{align}\begin{aligned}\label{eq:correction}
		\partial_t \wh{f}(t, \xi) &= \left[ \frac{\alpha i}{2t} \xi \left| \wh{f}(t, \xi) \right|^2 + \frac{\beta i}{2t} \xi \mathcal{H} \left( \left| \wh{f}(t, \cdot) \right|^2 \right)(\xi) \right] \wh{f}(t,\xi) + \text{l.o.t.}.
\end{aligned}\end{align}
 See Section \ref{sec:asymptotic} for detailed calculations. The resonance terms in \eqref{eq:correction} provide the precise phase corrections, which allow us to control the pointwise Fourier amplitude of the solution by eliminating the resonance cases.


In applying this approach, both the DNLS and non-local nonlinear terms have to be handled. Due to the non-locality, the standard gauge transformation for derivative NLS does not work. Specifically, for the nonlinear term $\partial_x (|u|^2 u)$, the high-low-low interaction term $u_{\text{low}} \overline{u_{\text{low}}} \partial_x u_{\text{high}}$ is a problematic case. However, this delicate case can be resolved in the case of $\beta = 0$ using the typical gauge transformation:
\begin{align}\label{eq:dnls-gaugetrans}
	u(t,x) \mapsto u(t,x) \exp\left[ \frac{1}{2i} \partial_x^{-1} (2\alpha |u(t)|^2) \right].
\end{align}
This transformation eliminates the problematic case when $\beta = 0$. In contrast, for $\beta \neq 0$, the transformation \eqref{eq:dnls-gaugetrans} cannot eliminate the term $\beta \mathcal{H}(\overline{u} \partial_x u ) u$. To address this issue, we introduce a frequency-localized gauge transformation (see \cite{Tao04,IK2007,KeTa06,KT-gauge}). A detailed discussion of this transformation is in Section \ref{sec:gauge}. 

Even though we use the frequency-localized gauge transformation to avoid the resonance case, the non-locality involving the anti-derivative causes singularities. This leads to terms which are difficult to handle by the perturbation argument based on Duhamel's formula. To overcome this obstacle, we apply the energy method, which gives cancellation effects to handle the singularity.


It is worth noting that our equation \eqref{kdnls} closely resembles the Calogero-Moser derivative NLS (CM-DNLS):
\begin{align*}
	i\p_t u +\p_x^2 u = 2i \p_x  P_+(|u|^2)u,
\end{align*}
which also contains both derivative and nonlocal nonlinearities. Here, $P_+$  is defined in Section \ref{sec:gauge} and we note that $2i\p_x P_+ = (i-\cH)\p_x$. CM-DNLS arises as a continuum model of the classical Calogero-Moser Hamiltonian system. Compared to our equation, this equation is integrable and enjoys an $L^2$ conservation law. Moreover, in the CM-DNLS, the derivative falls on the quadratic term and is supported on the half-line in the Fourier space. We refer the reader to \cite{gelen2024,KillLauVis2023} for the well-posedness results and \cite{hokowa2024,kimkimkwon2024} for the blowup results. Especially see \cite{kimkwon2024} for the soliton resolution.

This paper is organized as follows. In Section \ref{sec:gauge} we introduce the frequency-localized gauge transform for \eqref{kdnls} with $\beta \neq 0$. Section \ref{sec:lemmas} contains proofs of various useful estimates for the terms related to the gauge transformation and bilinear form. We also provide a  bootstrap assumption that leads to optimal dispersive decay for the solution to \eqref{kdnls} with $\beta \neq 0$.  Section \ref{sec:energy} is devoted to performing energy methods for our main equations, where we carefully handle delicate frequency relation and resonant cases. In Section \ref{sec:weighted}, to control the weighted norm in the bootstrap assumption, we begin with the properties of Galilean generator $J(t)= x +2it\p_x$. We carry out the energy methods similar to those of Section \ref{sec:energy} and close the bootstrap argument. In Section \ref{sec:asymptotic} we show the asymptotic behaviors of the solution, from which we can control the Fourier amplitude norm. We also describe how to extract the phase modification of the solution from the resonance interaction between cubic non-local nonlinear terms. Finally, in Section \ref{sec:main} we prove Theorem \ref{thm:main}.

\subsection*{Notations}
\noindent $\bullet$ (Mixed-normed spaces) For a Banach space $X$
and an interval $I$, $u\in L_{I}^{q}X$ iff $u(t)\in X$ for a.e.$t\in I$
and $\|u\|_{L_{I}^{q}X}:=\|\|u(t)\|_{X}\|_{L_{I}^{q}}<\infty$. Especially,
we denote $L_{I}^{q}L_{x}^{r}=L_{t}^{q}(I;L_{x}^{r}(\R))$, $L_{I,x}^{q}=L_{I}^{q}L_{x}^{q}$,
$L_{t}^{q}L_{x}^{r}=L_{\mathbb{R}}^{q}L_{x}^{r}$.

\noindent $\bullet$  As usual different positive constants are denoted  by the same letter $C$, if not specified.
$A\lesssim B$ and $A\gtrsim B$ means that $A\le CB$ and $A\ge C^{-1}B$,
respectively for some $C>0$. $A\sim B$ means that $A\lesssim B$
and $A\gtrsim B$.

\noindent $\bullet$ We fix our definition of the (one-dimensional) Fourier and inverse Fourier transforms as
\[ \cF_x[f](\xi )=\wh{f}(\xi) :=\frac{1}{\sqrt{2\pi}}\int_\R e^{-ix\xi}f(x)dx \;\;\mbox{ and }\;\; \cF^{-1}_\xi [g](x):=\frac{1}{\sqrt{2\pi}}\int_\R e^{ix\xi}g(\xi)d\xi . \]

\noindent $\bullet$  We use the notation $\bra{\cdot} := (1+ |\cdot|)^\frac12$.

\noindent $\bullet$ (Littlewood-Paley operators) Let $\varrho$ be a
Littlewood-Paley function such that $\varrho\in C_{0}^{\infty}(B(0,2))$
with $\varrho(\xi)=1$ for $|\xi|\le1$ and define $\varrho_{N}(\xi):=\varrho\left(\frac{\xi}{N}\right)-\varrho\left(\frac{2\xi}{N}\right)$
for $N\in2^{\mathbb{Z}}$. Then we define the frequency projection
$P_{N}$ by $\mathcal{F}(P_{N}f)(\xi)=\varrho_{N}(\xi)\widehat{f}(\xi)$,
and also $\varrho_{>N_0}:=\sum_{N>N_{0}}\varrho_{N}$ and  $\varrho_{\le N_{0}}:=1-\varrho_{>N_0}$. In addition
$P_{>N_0}:=\sum_{ N> N_{0}}P_{N}$, $P_{\le N_0}:= 1- P_{>N_0}$, and
$P_{\sim N_{0}}:=\sum_{N\sim N_{0}}P_{N}$. For $N\in2^{\mathbb{Z}}$
we denote $\widetilde{\varrho_{N}}=\varrho_{N/2}+\varrho_{N}+\varrho_{2N}$. In particular, $\widetilde{P_{N}}P_{N}=P_{N}\widetilde{P_{N}}=P_{N}$
where $\widetilde{P_{N}}=\mathcal{F}^{-1}\widetilde{\varrho_{N}}\mathcal{F}$.
Especially, we denote $P_{N}f$ by $f_{N}$ for any tempered distribution 
$f$.

\section{Gauge transformation}
\label{sec:gauge}

In this section, we define the gauge transformations which remove the worst nonlinear interactions of high-low type. 
They are non-periodic counterparts of the ones introduced in \cite{KT-gauge} for the periodic problem, treating three kinds of nonlinearities $\alpha |u|^2\partial_xu$, $\beta \mathcal{H}(|u|^2)\partial_xu$ and $\beta \mathcal{H}(\ol{u}\partial_xu)u$.

We use the operators $P_\pm:=\mathcal{F}^{-1}\chi_\pm \mathcal{F}$ and $Q_\pm :=P_\pm P_{>1}$, 
where $\chi_\pm$ is the characteristic function of the half line $\{ \xi:\pm \xi>0\}$.
Note that $\mathcal{H}=-iP_++iP_-$ and $\mathcal{H}Q_\pm =\mp iQ_\pm$.
To see the idea, let $u$ be a solution of \eqref{kdnls} and observe that
\begin{align*}
	Q_\pm \partial _x\big[ |u|^2u\big] &=2Q_\pm \big[ |u|^2\partial_xu\big] +Q_\pm \big[ u^2\partial_x\ol{u}\big] \\
	&=2|u|^2\partial_xQ_\pm u+\mathcal{R}_1^\pm ,\\
	Q_\pm \partial_x\big[ \mathcal{H}(|u|^2)u\big] &=Q_\pm \big[ \mathcal{H}(\ol{u}\partial_xu)u+\mathcal{H}(u\partial_x\ol{u})u+\mathcal{H}(|u|^2)\partial_xu\big] \\
	&=|u|^2\partial_xQ_\pm \mathcal{H}u+\mathcal{H}(|u|^2)\partial_xQ_\pm u +\mathcal{R}_2^\pm \\
	&=\big[ {\mp} i|u|^2 +\mathcal{H}(|u|^2)\big] \partial _xQ_\pm u +\mathcal{R}_2^\pm ,
\end{align*}
where 
\begin{align*}
	\mathcal{R}_1^\pm &:=Q_\pm \big[ u^2\partial_x\ol{u}\big] +2[Q_\pm ,|u|^2]\partial_xu,\\
	\mathcal{R}_2^\pm &:= Q_\pm \big[ \mathcal{H}(u\partial_x\ol{u})u\big] +[Q_\pm ,\mathcal{H}(|u|^2)]\partial_xu+ \Big( Q_\pm \big[ \mathcal{H}(\ol{u}\partial_xu)u\big] -|u|^2\partial_xQ_\pm \mathcal{H}u \Big) .
\end{align*}
Therefore, by defining 
\[
	P^\pm_{\alpha,\beta}(g)
	:=
	\left[
	\left(-i\alpha\mp\frac{\beta}{2}\right)\mathrm{Id}
	-i\frac{\beta}{2}\mathcal{H}
	\right]g.
\]
we have
\begin{align}
	(\partial _t-i\partial_x^2)Q_\pm u \ &=\ 2i \palbe(|u|^2)\partial _xQ_\pm u +\alpha \mathcal{R}_1^\pm +\beta \mathcal{R}_2^\pm .\label{eq:upm}
\end{align}
Note that 
\begin{align*}
		\palbe(|u|^2)&	=\Big[ \Big( {-}i\alpha \mp \frac{\beta}{2}\Big) \mathrm{Id}  -i\frac{\beta}{2}\mathcal{H}\Big] (|u|^2)=-i\alpha |u|^2 \mp \frac\beta2 |u|^2 -i\frac\beta2 \mathcal{H}(|u|^2).
\end{align*}
%
%
%
%
Now, the high-low interactions in $\mathcal{R}_1^\pm$ and $\mathcal{R}_2^\pm$ are not problematic.
Indeed, the first terms (those of $u^2\partial_x\ol{u}$ type) have good non-resonance properties; the second terms and the last term of $\mathcal{R}_2^\pm$ have no high-low interactions due to their commutator structure.
Then, we wish to treat the main term $2i \palbe(|u|^2)\partial _xQ_\pm u$ by applying suitable gauge transformations.
Note that we need to define the gauges for $Q_+u$ and $Q_-u$ separately when $\beta\neq 0$, since $P^+_{\alpha,\beta}(|u|^2)$ and $P^-_{\alpha,\beta}(|u|^2)$ are different. 
The previous results (see, e.g., \cite{Tao04, KeTa06}) suggest that the high-low interactions of the nonlinear term $F[u]\partial_xP_\pm u$ can be removed by the gauge transformation $P_\pm u\mapsto \exp (\frac{1}{2i}\partial_x^{-1}F[u])P_\pm u$.
Hence, we come to the following definition of our gauge transformations:
\begin{defn}[Gauge transformations]\label{def:gauge}
Let $\psi:\R \to \R$ be a compactly supported real-valued smooth function satisfying $\int_{\mathbb{R}}\psi (z)\,dz =1$.
Define the anti-derivative operators $\p_x^{-1}:L^1(\R)\to L^\infty(\R)$ and $\wt{\p}_x^{-1}:L^1_{\rm loc}(\R)\to W^{1,1}_{\rm loc}(\R)$ by
\[ \p_x^{-1}f(x):=\int_{-\infty}^xf(y)\,dy,\qquad \wt{\p}_x^{-1}g(x):=\int_{\R}\psi(z)\int_z^xg(y)\,dy\,dz \qquad (x\in \R).\]
Note that $\p_x\p_x^{-1}f=\p_x^{-1}\p_xf=f$, $\p_x\wt{\p}_x^{-1}g=g$ and $\wt{\p}_x^{-1}\p_xg=g-\int_{\mathbb{R}}\psi (z)g(z)\,dz$.
Then, define the gauge transformations $u(t)\mapsto v_\pm (t)$ by
\begin{gather*}
	v_\pm (t):=e^{\rho ^\pm [u(t)]}Q_{\pm } u(t), \qquad \rho ^\pm [u]:=\Big( {-}i\alpha \mp \frac\beta 2\Big) \p_x^{-1}\big[ |u|^2\big] -i\frac\beta 2 \wt{\p}_x^{-1}\big[ \mathcal{H}(|u|^2)\big] .
\end{gather*}
\end{defn}
\begin{rem}
In \cite{KL-LWP} we used simpler transformations
\begin{gather}\label{def:gauge0}
e^{\rho_0^\pm[u]}Q_\pm u,\qquad \rho_0^\pm[u]:={\rm Re}\, \rho^\pm[u] = \mp \frac{\beta}{2}\p_x^{-1}\big[ |u|^2\big] ,
\end{gather}
which treat only $\beta \mathcal{H}(\ol{u}\partial_xu)u$.
In the present work, we also need to treat $\alpha |u|^2\partial_xu$ and $\beta \mathcal{H}(|u|^2)\partial_xu$ to obtain an optimal decay in the energy estimates of weighted norms (see Remark~\ref{rem:gauge} below).
Also, we partially use the anti-derivative $\wt{\p}_x^{-1}g$ defined above in addition to the standard one $\p_x^{-1}f$.
It has been used when $g$ is not in $L^1$ (see, e.g., \cite{BuPl08}), which is our case as the term $\mathcal{H}(|u|^2)$ in $\palbe (|u|^2)$ may not be integrable due to unboundedness of the Hilbert transform on $L^1$.
Note that $\mathcal{H}(|u|^2)\in L^{1+}\subset L^1_{\rm loc}$ for $u\in H^{0+}$, by the Sobolev embedding and the boundedness of $\mathcal{H}$ on $L^p$, $p\in (1,\infty)$.
\end{rem}


We move on to the equation for $v_\pm$. We assume that the solutions are smooth.
Using the equation \eqref{eq:upm} and
denoting 
\begin{align*}
\mathcal{R}_3^\pm &:= \alpha \mathcal{R}_1^\pm +\beta \mathcal{R}_2^\pm \\
&\;=\alpha Q_{\pm} \big[ u^2\partial_x\ol{u}\big] + \beta Q_{\pm} \big[ \mathcal{H}(u\partial_x\ol{u})u\big] + \beta \Big( Q_{\pm} \big[ \mathcal{H}(\ol{u}\partial_xu)u\big] -|u|^2\partial_xQ_{\pm}\mathcal{H}u \Big) \\
&~\quad +[Q_{\pm} ,2\alpha |u|^2+\beta \mathcal{H}(|u|^2) ]\partial_xu,
\end{align*}
we have 
\begin{align*}
	(\partial _t-i\partial_x^2)v_\pm &= e^{\rho ^\pm [u]}\Big[ 2i\palbe(|u|^2)Q_\pm \p_xu +\mathcal{R}_3^\pm -2i (\p_x \rho ^\pm) Q_{\pm} \p_x u \\
	&\hspace{5cm} +(\p_t \rho ^\pm -i \p_x^2\rho ^\pm) Q_{\pm} u-i(\p_x \rho ^\pm )^2 Q_{\pm} u\Big] .
\end{align*}
Moreover, we have
\begin{gather*}
	\p_x \rho ^\pm =\palbe (|u|^2),\qquad \p_x^2\rho ^\pm = \palbe (\ol{u}\p_xu+u\p_x\ol{u}),
\end{gather*}
and from the equation \eqref{kdnls},
\begin{align*}
\p_t(|u|^2)&=2\mathrm{Re}\, \Big\{ \Big( i\p_x^2u+\alpha \partial_x \big[ |u|^2u\big] +\beta \partial _x\big[ \mathcal{H}(|u|^2)u\big] \Big) \ol{u}\Big\} \\
&=\partial _x\Big(2\mathrm{Re}\,(i\ol{u}\p_xu)+\frac{3}{2}\alpha |u|^4+2\beta \mathcal{H}(|u|^2)|u|^2\Big) -\beta \mathcal{H}(|u|^2)\partial_x(|u|^2),
\end{align*}
which implies
\begin{gather*}
	\begin{aligned}
		\p_t\rho ^\pm
		&=\Big( {-}i\alpha \mp \frac\beta 2\Big) \p_x^{-1}\big[ \p_t (|u|^2)\big]  -i\frac\beta 2 \wt{\p}_x^{-1}\big[ \mathcal{H}(\p_t(|u|^2))\big] \\
		&=iP^\pm _{\alpha,\beta} (\ol{u}\p_xu-u\p_x\ol{u}) +P^\pm _{\alpha,\beta}\Big[ \frac{3}{2}\alpha |u|^4+2\beta \mathcal{H}(|u|^2)|u|^2\Big] \\
		&\quad +\beta \Big( i\alpha \pm \frac{\beta}{2}\Big) \p_x^{-1}\big[ \mathcal{H}(|u|^2)\p_x(|u|^2)\big] \\
		&\quad +i\frac{\beta}{2}\int_{\mathbb{R}}\psi(z)\mathcal{H}\Big[ 2\mathrm{Re}\,(i\ol{u}\p_xu)+\frac{3}{2}\alpha |u|^4+2\beta \mathcal{H}(|u|^2)|u|^2\Big] (z)\,dz \\
		&\quad +i\frac{\beta^2}{2} \wt{\p}_x^{-1}\mathcal{H}\big[ \mathcal{H}(|u|^2)\partial_x(|u|^2)\big] .
	\end{aligned}
\end{gather*}
Considering the cancellation effect in $\p_t\rho ^\pm -i\p_x^2\rho ^\pm $, the worst term $P^\pm_{\alpha ,\beta}(\ol{u}\p_xu)$ disappears and instead we have $-2iP^\pm_{\alpha ,\beta}(u\p_x\ol{u})$,
\begin{align*}
\p_t\rho ^\pm -i\p_x^2\rho ^\pm &= -2\alpha u\p_x\ol{u}  \pm i \beta  u\p_x\ol{u} -\beta \mathcal{H}(u\p_x\ol{u}) \\
		&\quad +i\frac{\beta^2}{2}\wt{\p}_x^{-1}\mathcal{H}\big[ \mathcal{H}(|u|^2)\partial_x(|u|^2)\big] \\
		&\quad +i\frac{\beta}{2} \int_{\mathbb{R}}\psi(z)\mathcal{H}\Big[ 2\mathrm{Re}\,(i\ol{u}\p_xu)+\frac{3}{2}\alpha |u|^4+2\beta \mathcal{H}(|u|^2)|u|^2\Big] (z)\,dz \\
		&\quad +\mathcal{R}_4^\pm ,
\end{align*}
where
\begin{align*}
	\cR_4^\pm &:= P^\pm _{\alpha,\beta}\Big[ \frac{3}{2}\alpha |u|^4 +2\beta \mathcal{H}(|u|^2)|u|^2\Big] +\beta \Big( i\alpha \pm \frac{\beta}{2}\Big) \p_x^{-1}\big[ \mathcal{H}(|u|^2)\partial_x(|u|^2)\big] .
\end{align*}
Thus we have
\begin{align}\label{eq:v-gauge-transform}
	\begin{aligned}
		(\partial _t-i\partial_x^2)v_\pm &=e^{\rho ^\pm [u]}\mathcal{R}_5^\pm +G_\pm [u] v_\pm  \\
		&\quad +i\frac{\beta^2}{2}\wt{\p}_x^{-1}\mathcal{H}\big[ \mathcal{H}(|u|^2)\partial_x(|u|^2)\big] \cdot v_\pm \\ 
		&\quad +i\frac{\beta}{2} \int_{\mathbb{R}}\psi(z)\mathcal{H}\Big[ 2\mathrm{Re}\,(i\ol{u}\p_xu)+\frac{3}{2}\alpha |u|^4+2\beta \mathcal{H}(|u|^2)|u|^2\Big] (z)\,dz \cdot v_\pm , 
	\end{aligned}
\end{align}
where
\begin{align}\label{eq:def-r5}
\begin{aligned}
	\mathcal{R}_5^\pm&:=\mathcal{R}_3^\pm +\Big( {-}2\alpha u\p_x\ol{u}  \pm i \beta  u\p_x\ol{u} -\beta \mathcal{H}(u\p_x\ol{u}) \Big) Q_\pm u \\
&\;=\alpha \Big( Q_\pm \big[ u^2\p_x\ol{u}\big] - 2u\p_x\ol{u} Q_\pm u \Big) \\
&\quad + \beta \Big( Q_\pm \big[ \mathcal{H}(u\p_x\ol{u})u\big] -(Q_\pm\mathcal{H}u)(\p_x \ol{u}) u - \mathcal{H}(u\p_x\ol{u}) Q_\pm u \Big) \\
&\quad +\beta \Big( Q_\pm \big[ u\mathcal{H}(\ol{u}\p_xu)\big]  -|u|^2Q_\pm \mathcal{H} \p_xu \Big)  \\
&\quad +[Q_\pm ,2\alpha |u|^2 +\beta \mathcal{H}(|u|^2)]\partial_xu
\end{aligned}
\end{align}
and
\begin{align}\label{eq:def-gpm}
\begin{aligned}
	G_\pm [u]&:=\mathcal{R}_4^\pm -i\big[ P_{\alpha,\beta}^\pm (|u|^2)\big] ^2 \\
&\;= P^\pm _{\alpha,\beta}\Big[ \frac{3}{2}\alpha |u|^4+2\beta \mathcal{H}(|u|^2)|u|^2\Big] -i\big[ P_{\alpha,\beta}^\pm (|u|^2)\big] ^2+\beta \Big( i\alpha \pm \frac{\beta}{2}\Big) \p_x^{-1}\big[ \mathcal{H}(|u|^2)\partial_x(|u|^2)\big] .
\end{aligned}
\end{align}

\section{Preliminaries}\label{sec:lemmas}
In this section, we collect several lemmas that are used throughout this paper. We begin with a property of the linear Schr\"odinger equation.
\begin{lemma}\label{lem:linear}
	The Schr\"odinger group satisfies the inequalities
	\begin{align}\label{eq:decay-sch}
	\begin{aligned}
	\left\|e^{it\p_x^2}g \right\|_{L^\infty} &\les t^{-\frac12} \|\wh{g}\|_{L_\xi^\infty} + t^{-\frac12 -\gamma} \|g\|_{H^{0,\gamma'}},\\
	\left\|e^{it\p_x^2}g \right\|_{W^{1,\infty}} &\les t^{-\frac12} \|\bra{\xi}\wh{g}\|_{L_\xi^\infty} + t^{-\frac12 -\gamma} \|g\|_{H^{1,\gamma'}}
	\end{aligned}
	\end{align}
for any $\gamma, \gamma' >0$ satisfying $\gamma' > \frac12 + 2\gamma$.
\end{lemma}

\begin{proof}
This follows from the identity
\begin{align*}
e^{it\partial_x^2}g(x)&=\frac{1}{(4\pi it)^{1/2}}\int_{\mathbb{R}}e^{i\frac{(x-y)^2}{4t}}g(y)\,dy \\
&=\frac{1}{(4\pi it)^{1/2}}e^{i\frac{x^2}{4t}}\left( \int_{\mathbb{R}}e^{-i\frac{x}{2t}y}g(y)\,dy +\int_{\mathbb{R}}e^{-i\frac{x}{2t}y} (e^{i\frac{y^2}{4t}}-1) g(y)\,dy\right) \\
&=\frac{1}{(2it)^{1/2}}e^{i\frac{x^2}{4t}}\wh{g}\left( \frac{x}{2t}\right) +O\left( t^{-\frac12-\gamma}\| |y|^{2\gamma}g(y)\|_{L^1_y}\right) ,\qquad \gamma>0
\end{align*}
and Cauchy-Schwarz inequality in $y$.
\end{proof}

%

\subsection{Fundamental estimates and bootstrap assumption}
We set a bootstrap assumption
 \begin{align}\label{eq:assumption-apriori}
	\|u\|_{X_T} &:= \sup_{t \in [0,T]} \left( \bra{t}^{-\de} \|u(t)\|_{H^2} + \bra{t}^{-2\de} \|xf(t)\|_{H^1} \right) \le \ve_1
\end{align}  
for some  $T, \ve_1 ,\delta >0$.
We always assume $\ve\leq \ve_1\leq 1$, where $\ve$ is the size of initial data in \eqref{thm:initial}, and that $\delta$ is sufficiently small.

Under the bootstrap assumption \eqref{eq:assumption-apriori}, we have the following proposition.
\begin{prop}\label{prop:linfty}
	Assume that $u \in C([0,T], H^{2}\cap H^{1,1})$ is a solution to \eqref{kdnls} satisfying \eqref{eq:assumption-apriori}. 
	Then, for $0\leq t\leq T$ we have
	\begin{align*}
		\left\|\bra{\xi}\wh{f}(t,\xi)\right\|_{L_\xi^\infty}  \les \ve_1,
	\end{align*}
	where $f(t)$ is defined in  \eqref{eq:correction}.	Furthermore, there exists $\kappa>0$ such that for any dyadic $N\geq 1$ and $0\leq t\leq T$ we have
	\begin{equation*}
		\left\| \varrho_N(\xi) \bra{\xi}\wh{f}(t,\xi)\right\|_{L_\xi^\infty}  \les N^{-\kappa}\ve_1.
	\end{equation*}
\end{prop}

\begin{rem}\label{rem:hayashi-naumkin} Proposition \ref{prop:linfty} will be proved in Section \ref{sec:asymptotic} as part of Lemma \ref{lem:asymptotic}. We point out that Lemma \ref{lem:asymptotic} gives the core estimates of the Hayashi-Naumkin method. Indeed, such estimates were first established by Hayashi and Naumkin to verify modified scattering behaviors for nonlinear Schr\"odinger equations with critical nonlinearities \cite{hayashi-naumkin1998} and for the derivative NLS \cite{HN1998-dnls,HN1997-dnls}. 
\end{rem}
Combining the estimates in Proposition~\ref{prop:linfty} with Lemma \ref{lem:linear}, we have the crucial decay estimates
\begin{align}
\|u(t)\|_{W^{1,\infty}} &\les \bra{t}^{-\frac12} \ve_1, \label{eq:linear} \\
\| P_Nu(t)\|_{W^{1,\infty}} &\les N^{-\kappa}\bra{t}^{-\frac12} \ve_1\quad (N\geq 1). \label{eq:linear-local}
\end{align}
To show \eqref{eq:linear-local}, we follow the proof of Lemma~\ref{lem:linear} with $g=P_N\bra{\p_x}f(t)$. Indeed, for $t\geq 1$ we estimate the second term on the right-hand side of \eqref{eq:decay-sch}, taking  $0< \gamma< \frac18$ such that $\gamma\geq \frac74\delta$, as
\begin{align*}
&t^{-\frac12-\gamma}\| |x|^{2\gamma}P_N\bra{\p_x}f(t)\|_{L^1}\\
&\les t^{-\frac12-\gamma}\| \bra{x}^{\frac34}P_N\bra{\p_x}f(t)\|_{L^2}\\
&\les t^{-\frac12-\gamma}\Big( \| xP_N \bra{\p_x}f(t)\|_{L^2}^{\frac34}\| P_Nf(t)\|_{H^1}^{\frac14}+\| P_Nf(t)\|_{H^1}\Big) \\
&\les t^{-\frac12-\gamma}\Big\{ \Big( \| xf(t)\|_{H^1}+\| [x,P_N\bra{\p_x}]f(t)\|_{L^2}\Big) ^{\frac34}\Big( N^{-1}\| u(t)\|_{H^2}\Big) ^{\frac14} +N^{-1}\| u(t)\|_{H^2}\Big\} \\
&\les N^{-\frac14}t^{-\frac12-\gamma+\frac74\de}\ve_1+N^{-1}t^{-\frac12-\gamma+\delta}\ve_1 \les N^{-\frac14}t^{-\frac12}\ve_1,
\end{align*}
where we have used
\[
 \| [x,P_N\bra{\p_x}]f(t)\|_{L^2}= \left\| \Big[\p_\xi, \varrho_N(\xi)\bra{\xi}\Big]\wh{f}(t) \right\|_{L^2}\lesssim \| \wh{f}(t)\|_{L^2}.
 \]
Note that the case $t\in [0,1]$ follows directly from the embedding $H^{\frac53}\hookrightarrow W^{1,\infty}$, for instance.

The localized decay estimate \eqref{eq:linear-local} with a decay in $N$ will also be useful throughout the proof. Here, we mention that \eqref{eq:linear-local} can be used to verify that
\begin{align}\label{eq:linear-ppm}
\| Q_\pm u(t)\|_{W^{1,\infty}}\les \bra{t}^{-\frac12}\ve_1.
\end{align}
In fact, since $Q_\pm\wt{P}_N$ is bounded on $L^\infty$ uniformly in $N$, by \eqref{eq:linear-local} we can estimate as
\[ \| Q_\pm u(t)\|_{W^{1,\infty}}\leq \sum _{N\geq 1}\| P_Nu(t)\|_{W^{1,\infty}}\les \sum_{N\geq 1}N^{-\kappa}\bra{t}^{-\frac12}\ve_1\les \bra{t}^{-\frac12}\ve_1.\]

The next two lemmas are important consequences of the $L^\infty$ decay \eqref{eq:linear}.
\begin{lemma}\label{lem:esti-null}
It holds that
\[ \p_x(u(t)\ol{v(t)})=\frac{1}{2it}\Big( \ol{v(t)}e^{it\p_x^2}xf(t)-u(t)\ol{e^{it\p_x^2}xg(t)}\Big) ,\]
where $f(t):=e^{-it\p_x^2}u(t)$ and $g(t):=e^{-it\p_x^2}v(t)$.
In particular, if $u$ satisfies \eqref{eq:assumption-apriori}, then we have
\begin{align}\label{eq:esti-null}
	\| \p_x(|u(t)|^2) \|_{H^1} \les \bra{t}^{-\frac32 +2\de}\ve_1^2.
\end{align}
\end{lemma}
\begin{proof}
Although this is a well-known identity $2it\p_x(u\ol{v})=(Ju)\ol{v}-u(\ol{Jv})$ related to the operator $J:=x+2it\p_x$, it is worth noticing that it can be shown by an integration by parts on the Fourier side:
\begin{align*}
	\p_x (u\ol{v}) 
	&= -\frac{1}{2t}\mathcal F_\xi^{-1} \left[ \frac{1}{\sqrt{2\pi}}\int_{\R}  \p_\eta \big[ e^{-it[(\xi+\eta)^2 - \eta^2]}\big] \wh{f}(\xi+\eta) \ol{\wh{g}(\eta)}    d\eta \right]\\
	&= \frac1{2t}\mathcal F_\xi^{-1} \left[ \frac{1}{\sqrt{2\pi}}\int_{\R}   e^{-it[(\xi+\eta)^2 - \eta^2]} \p_\eta \left(\wh{f}(\xi+\eta) \ol{\wh{g}(\eta)} \right)    d\eta \right] \\
	&= \frac1{2it}\left(\ol{v}\, e^{it\p_x^2}xf - u\, \ol{e^{it\p_x^2}xg}\right) .
\end{align*}
Taking $u=v$ and using H\"older's inequality, \eqref{eq:assumption-apriori} and \eqref{eq:linear}, we see that
\[ \|\p_x(|u(t)|^2)\|_{H^1}\les \frac1t \|u(t)\|_{W^{1,\infty}} \|xf(t)\|_{H^1}\les t^{-1} \bra{t}^{-\frac12+2\de}\ve_1^2, \]
which shows \eqref{eq:esti-null} for $t\geq 1$.
The case $t\in [0,1]$ follows from the direct estimate $\| \p_x(|u|^2)\|_{H^1}\les \| u\|_{H^2}^2$.
\end{proof}

\begin{lemma}[$L^2$ energy estimate]\label{lem:L2energy}
Assume that $u \in C([0,T], H^{2}\cap H^{1,1})$ is a solution to \eqref{kdnls} satisfying  \eqref{eq:assumption-apriori}.
Then, for $0\leq t\leq T$ we have the following: If $\beta >0$, then $\| u(t)\|_{L^2}$ is nondecreasing and 
\[ \| u(0)\|_{L^2}\leq \| u(t)\|_{L^2}\leq \min \big\{ \| u(0)\|_{L^2}+C\ve_1^3,\, \| u(0)\|_{L^2}e^{C\ve_1^2}\big\} .\]
If $\beta <0$, then $\| u(t)\|_{L^2}$ is nonincreasing and 
\[ \| u(0)\|_{L^2}\geq \| u(t)\|_{L^2}\geq \max \big\{ \| u(0)\|_{L^2}-C\ve_1^3,\, \| u(0)\|_{L^2}e^{-C\ve_1^2}\big\} .\]
In particular, in both cases we have $\| u(t)\|_{L^2}\sim \| u(0)\|_{L^2}$. 
\end{lemma}
\begin{proof}
By the equation \eqref{kdnls} we have
\begin{align*}
		\frac{d}{dt}\| u(t)\|_{L^2}^2=\beta \int_{\R} D_x(|u(t)|^2)|u(t)|^2 dx= \beta \big\| D_x^{\frac12}(|u(t)|^2)\big\|_{L^2}^2,
\end{align*}
from which monotonicity of the $L^2$ norm follows.
By \eqref{eq:assumption-apriori}, \eqref{eq:linear}, and \eqref{eq:esti-null}, we estimate 
\begin{align*}
&\begin{aligned}
	\int_{\R} D_x(|u(t)|^2)|u(t)|^2 dx &\leq \normo{D_x(|u(t)|^2)}_{L^2}\normo{|u(t)|^2}_{L^2}\\
	& \les \bra{t}^{-\frac32 +2\de}\ve_1^2 \|u(t)\|_{L^\infty} \|u(t)\|_{L^2}\les \bra{t}^{-2+2\de}\ve_1^3\|u(t)\|_{L^2},
\end{aligned}
\intertext{and}
&\begin{aligned}
	\int_{\R} D_x(|u(t)|^2)|u(t)|^2 dx &\leq \normo{D_x(|u(t)|^2)}_{L^\infty}\normo{|u(t)|^2}_{L^1}\\
	& \les \normo{D_x(|u(t)|^2)}_{H^1}\normo{u(t)}_{L^2}^2\les \bra{t}^{-\frac32 +2\de}\ve_1^2 \|u(t)\|_{L^2}^2.
\end{aligned}
\end{align*}
This gives us that
\begin{align*}
	\frac{d}{dt}\| u(t)\|_{L^2} \les \beta \bra{t}^{-2+2\de}\ve_1^3,\;\; \mbox{ and } \;\;  \frac{d}{dt}\| u(t)\|_{L^2}^2\les \beta \bra{t}^{-\frac32+2\de}\ve_1^2\| u(t)\|_{L^2}^2
\end{align*}
when $\beta >0$, and that
\begin{align*}
	\frac{d}{dt}\| u(t)\|_{L^2} \gtrsim -|\beta| \bra{t}^{-2+2\de}\ve_1^3, \;\; \mbox{ and } \;\;   \frac{d}{dt}\| u(t)\|_{L^2}^{-2}\les |\beta| \bra{t}^{-\frac32+2\de}\ve_1^2\| u(t)\|_{L^2}^{-2}
\end{align*}
when $\beta <0$.
In each case, the desired a priori bounds are obtained by integrating the first inequality from $0$ to $t$ or applying Gronwall's lemma to the second inequality.
\end{proof}

\subsection{Useful estimates}
In this section, we introduce some useful lemmas that will be crucial ingredients in the main proof.

\begin{lemma}\label{lem:HLinfty}
Let $u$ satisfy   \eqref{eq:assumption-apriori}. Then for any $\zeta>0$, we have 
\[ \| \mathcal{H}(|u|^2)\|_{L^\infty}\les_\zeta \bra{t}^{-1+\zeta}\ve_1^2. \]
\end{lemma}
\begin{proof}
Since the case $t\in [0,1]$ can be done readily, we consider $t\geq 1$. For any positive integer $m$, we have
\begin{align*}
\| \mathcal{H}(|u|^2) \|_{L^\infty}^m=\| \mathcal{H}(|u|^2)^m\|_{L^\infty} &\les_m \| \mathcal{H}(|u|^2)^m\|_{L^2}^{\frac12}\| \mathcal{H}(|u|^2)^{m-1}\p_x\mathcal{H}(|u|^2)\|_{L^2}^{\frac12} \\
&\les \| \mathcal{H}(|u|^2)\|_{L^{2m}}^{m-\frac12}\| \p_x\mathcal{H}(|u|^2)\|_{L^{2m}}^{\frac12} \les_m \| |u|^2\|_{L^{2m}}^{m-\frac12}\| \p_x(|u|^2)\|_{L^{2m}}^{\frac12} \\
&\les \Big( \| u\|_{L^\infty}^{1-\frac1{2m}}\| u\|_{L^2}^{\frac1{2m}}\Big) ^{2m-1}\| \p_x(|u|^2)\|_{H^1}^{\frac12} \\
&\les (\bra{t}^{-\frac12(1-\frac1{2m})}\ve_1)^{2m-1}(\bra{t}^{-\frac32+2\delta}\ve_1^2)^{\frac12} \\
&= (\bra{t}^{-1+\frac{1+4\delta}{4m}-\frac{1}{4m^2}}\ve_1^2)^m.
\end{align*}
We finish the proof by taking $m$ sufficiently large.
\end{proof}%

\begin{lemma}\label{lem:rho}
Let $u$ satisfy  \eqref{eq:assumption-apriori}. Then we have
\begin{align}\label{eq:esti-rho}
	|e^{\rho^\pm[u]}| \les 1.
\end{align}
Moreover, we get
\begin{align}\label{eq:esti-rho-deri}
\left| \p_x e^{\rho^\pm[u]} \right| \les \bra{t}^{-1+\zeta}\ve_1^2,\qquad \left| \p_x^2 e^{\rho^\pm[u]} \right| \les \bra{t}^{-\frac32+2\delta}\ve_1^2,\qquad \left\| \p_x^3 e^{\rho^\pm[u]} \right\|_{L^2} \les \bra{t}^{-\frac32+2\delta}\ve_1^2
\end{align}
for any $\zeta >0$ and
\begin{align}\label{eq:esti-rho-derit}
\left| \p_te^{2\rho_0^\pm[u]}\right| \les \left| \p_t\rho_0^\pm[u]\right| \les \bra{t}^{-1}\ve_1^2.
\end{align}%
\end{lemma}  
\begin{proof}
To obtain \eqref{eq:esti-rho}, it suffices to consider the real part $\rho_0^\pm[u]$. Indeed, we have
\begin{align*}
\big| \p_x^{-1}(|u|^2)(x)\big| \leq \|u\|_{L^2}^2 \les \ve_1^2.
\end{align*}
Then we finish the proof of \eqref{eq:esti-rho}.
The bounds \eqref{eq:esti-rho-deri} follow from \eqref{eq:esti-rho}, Lemma~\ref{lem:HLinfty}, and
\[
 \| \p_x P^\pm_{\alpha,\beta}(|u|^2)\|_{L^\infty}+\| \p_x^2 P^\pm_{\alpha,\beta}(|u|^2)\|_{L^2}\les \| \p_x(|u|^2)\|_{H^1}\les \bra{t}^{-\frac32+2\delta}\ve_1^2
 \]
by Lemma~\ref{lem:esti-null}.
Finally, we obtain the following estimates
\begin{align*}
\left| \p_t\rho_0^\pm[u]\right| 
&=\left| \frac{\beta}{2} \int_{-\infty}^x\Big\{ \p_x\Big( 2{\rm Re}(i\ol{u}\p_xu)+\frac32\alpha|u|^4+2\beta \mathcal{H}(|u|^2)|u|^2\Big) -\beta \mathcal{H}(|u|^2)\p_x(|u|^2)\Big\} (y)\,dy \right| \\
&\les \left\| 2{\rm Re}(i\ol{u}\p_xu)+\frac32\alpha|u|^4+2\beta \mathcal{H}(|u|^2)|u|^2\right\|_{L^\infty}+\left\| \mathcal{H}(|u|^2)\p_x(|u|^2)\right\|_{L^1}\\
&\les \bra{t}^{-1}\ve_1^2+\bra{t}^{-2}\ve_1^4+\bra{t}^{-2+\zeta}\ve_1^4+\bra{t}^{-2+2\delta}\ve_1^4 \les \bra{t}^{-1}\ve_1^2,
\end{align*}
which together with \eqref{eq:esti-rho} implies \eqref{eq:esti-rho-derit}.
\end{proof}

\begin{cor}\label{cor:apriori-v}
Suppose that $u$ satisfies  \eqref{eq:assumption-apriori}.
Then we have
 \begin{align}\label{eq:esti-v}
 	\| v_\pm (t)\|_{H^2}\les \bra{t}^{\delta}\ve_1,\qquad \| xe^{-it\p_x^2}v_\pm(t)\|_{H^1}\les \bra{t}^{2\de}\ve_1,
 \end{align}
for $t \in [0,T]$.
\end{cor}

\begin{proof}
The first bound in \eqref{eq:esti-v} follows easily from Lemma~\ref{lem:rho} as follows:
\[ \| v_\pm \|_{H^2}\les \| e^{\rho^\pm[u]}\|_{W^{2,\infty}}\| Q_\pm u\|_{H^2}\les \| u\|_{H^2}\leq \bra{t}^{\de}\ve_1.\]
To prove the second estimate in \eqref{eq:esti-v}, we use the operator $J(t)=x+2it\p_x$ (see Section~\ref{sec:weighted} for the detailed properties). By
\[ Jv_\pm =2it\p_xe^{\rho^\pm[u]}\cdot Q_\pm u+e^{\rho^\pm [u]}[J,Q_\pm ]u+e^{\rho^\pm[u]}Q_\pm Ju ,\]
we have
\begin{align*}
\| xe^{-it\p_x^2}v_\pm\|_{H^1}&=\| Jv_\pm \|_{H^1}\les \bra{t}\| \p_xe^{\rho^\pm[u]}\|_{W^{1,\infty}}\| u\|_{H^1}+\| e^{\rho^\pm[u]}\|_{W^{1,\infty}}\big( \| u\|_{H^1}+\| Ju\|_{H^1}\big) \\
&\les \bra{t}^{\zeta+\de} \ve_1+\bra{t}^{2\de}\ve_1\les \bra{t}^{2\de}\ve_1,
\end{align*}
for   $0 < \zeta < \de$.
\end{proof}

We next give estimates for  $\mathcal{R}^\pm_5$ and $G_\pm[u]$, appearing in \eqref{eq:def-r5} and \eqref{eq:def-gpm}. To this end, we use the following commutator estimates.
\begin{lemma} \label{lem:commutator}
	Let $k\geq 1$ be an integer. Then we have
	\begin{align} \label{eq:comm-all}
		\| [Q_\pm ,f]g\|_{L^2}\les \min \left\{ \| f\|_{L^2}\big( \| g\|_{L^\infty}+\| Q_\pm g\|_{L^\infty}\big),~ \| \p_xf\|_{L^2}\| \wh{g}\|_{L^1_\xi} \right\}
	\end{align}
	and
	\begin{align} \label{eq:comm-high}
		\| \p_x^k([Q_\pm ,f]g)\|_{L^2}\les \| \p_xf\|_{H^{k-1}}\| \wh{g}\|_{L^1_\xi}.
	\end{align}
\end{lemma}
\begin{proof}
	The first bound in \eqref{eq:comm-all} follows directly from the H\"older's inequality. The mean value theorem implies
	\[ |\varrho_{>1}(\xi)\chi_\pm (\xi)-\varrho_{>1}(\eta)\chi_\pm (\eta)|\les |\xi-\eta| \]
	uniformly in $\xi,\eta\in \R$. Therefore,
	\[ \left\| \int_{\R}\big( \varrho_{>1}(\xi)\chi_\pm (\xi)-\varrho_{>1}(\eta)\chi_\pm (\eta)\big) \wh{f}(\xi-\eta)\wh{g}(\eta)\,d\eta \right\|_{L^2_\xi}\les \| \xi \wh{f}\|_{L^2_\xi}\| \wh{g}\|_{L^1_\xi},\]
	which implies the second bound in \eqref{eq:comm-all}. In view of this bound, \eqref{eq:comm-high} follows once we show
	\[ \| P_{\geq 8}\p_x^k([Q_\pm ,f]g)\|_{L^2}\les \| \p_x^kf\|_{L^2}\| \wh{g}\|_{L^1_\xi}.\]
	To prove the estimate above, we assume $|\xi |\geq 4$. Since $\varrho_{>1} \chi_\pm$ is constant on each of the support $\pm [2,\infty)$, we have
	\[ \varrho_{>1}(\xi)\chi_\pm (\xi)-\varrho_{>1}(\eta)\chi_\pm (\eta)\neq 0\quad \Longrightarrow \quad |\xi-\eta|\geq |\xi|-2\geq \frac12|\xi|.\]
	In particular, $|\xi|\les |\xi-\eta|$ if $\varrho_{>1}(\xi)\chi_\pm (\xi)-\varrho_{>1}(\eta)\chi_\pm (\eta)\neq 0$ and $|\xi|\geq 4$. This implies
	\[ \| P_{\geq 8}\p_x^k([Q_{\pm},f] g)\|_{L^2}\les \left\| \int_{\R}|\xi-\eta|^k|\wh{f}(\xi-\eta)||\wh{g}(\eta)|\,d\eta\right\|_{L^2_\xi}\les \| \xi^k\wh{f}\|_{L^2_\xi}\| \wh{g}\|_{L^1_\xi},\]
	showing \eqref{eq:comm-high}.
\end{proof}

\begin{lemma}\label{lem:R5}
	Suppose that $u$ satisfies  \eqref{eq:assumption-apriori}.
	Then we have
	\[ \| \mathcal{R}^\pm_5\|_{H^1}\les \bra{t}^{-1+\de}\ve_1^3.\]
\end{lemma}
\begin{proof}
	An appropriate use of H\"older's inequality (eliminating $\cH$ in the $L^2$ norm) together with \eqref{eq:linear} and \eqref{eq:linear-ppm} gives the bound $\| u\|_{W^{1,\infty}}^2\| u\|_{H^2}\les \bra{t}^{-1+\de}\ve_1^3$. This completes the proof for the most cases. 	The only term which needs to care is $\| \p_x([Q_\pm ,\cH(|u|^2)]\p_xu)\|_{L^2}$, and we exploit its commutator structure.
	Using Lemmas~\ref{lem:commutator} and \ref{lem:esti-null}, we estimate 
	\[ \|  \p_x ([Q_\pm ,\cH(|u|^2)]\p_xu)\|_{L^2}\les \| \p_x\cH(|u|^2)\|_{L^2}\| \p_xu\|_{H^1}\les \bra{t}^{-\frac32+3\de}\ve_1^3,\]
which completes the proof.
\end{proof}

\begin{lemma}\label{lem:G}
Suppose that $u$ satisfies  \eqref{eq:assumption-apriori}.
Then we have
\begin{align}
\| G_\pm [u]\|_{L^\infty}&\lesssim \bra{t}^{-2+2\delta}\ve_1^4,\label{eq:esti-g}\\
\| \p_xG_\pm [u]\|_{H^1}&\lesssim \bra{t}^{-\frac52+\zeta+2\de}\ve_1^4,\label{eq:esti-dg}
\end{align}
for any $0<\zeta <\delta$ and $t \in [0,T]$.
\end{lemma}
\begin{proof}
Let $\mathcal{A}_0,\mathcal{A}_1,\mathcal{A}_2 \in \big\{{\rm Id},\mathcal{H},P^\pm_{\alpha,\beta}\big\}$.
By \eqref{eq:linear} and Lemmas~\ref{lem:esti-null}, \ref{lem:HLinfty}, we have
\begin{align*}
\| \mathcal{A}_1(|u|^2)\mathcal{A}_2(|u|^2)\|_{L^2}&\les \| |u|^2\|_{L^4}^2\leq \| u\|_{L^\infty}^3\| u\|_{L^2}\les \bra{t}^{-\frac32}\ve_1^4,\\
\| \mathcal{A}_1(|u|^2)\p_x\mathcal{A}_2(|u|^2)\|_{H^1}&\les \| \mathcal{A}_1(|u|^2)\|_{W^{1,\infty}}\| \p_x(|u|^2)\|_{H^1}\les \bra{t}^{-\frac52+\zeta+2\delta}\ve_1^4.
\end{align*}
The bound \eqref{eq:esti-dg} follows directly from the second estimate.
To prove \eqref{eq:esti-g}, we use these estimates as
\[ \big\| \mathcal{A}_0\big[ \mathcal{A}_1(|u|^2)\mathcal{A}_2(|u|^2)\big] \big\|_{L^\infty}\les \| \mathcal{A}_1(|u|^2)\mathcal{A}_2(|u|^2)\|_{L^2}^{\frac12}\big\| \p_x\big[ \mathcal{A}_1(|u|^2)\mathcal{A}_2(|u|^2)\big]\big\|_{L^2}^{\frac12}\les \bra{t}^{-2+\frac{\zeta}{2}+\delta}\ve_1^4,\]
together with 
\[ \Big\| \int_{-\infty}^x\big[ \mathcal{H}(|u|^2)\p_x(|u|^2)\big] (y)\,dy \Big\|_{L^\infty} \leq \| \mathcal{H}(|u|^2)\p_x(|u|^2)\|_{L^1}\les \| |u|^2\|_{L^2}\| \p_x(|u|^2)\|_{L^2}\les \bra{t}^{-2+2\de}\ve_1^4. \]
Since $\zeta <\delta$, this implies \eqref{eq:esti-g}.
\end{proof}

Finally, we give estimates related to the time derivative of $f(t)$, which will be needed when we exploit the time non-resonance property by a normal form approach.
\begin{lemma}\label{lem:derivativ-f}
	Let $u\in C([0,T];H^2\cap H^{1,1})$ be a solution of \eqref{kdnls} satisfying the bootstrap assumption \eqref{eq:assumption-apriori} and $f(t)= e^{-it\p_x^2}u(t)$. Then we have, for any $\zeta>0$,	
	\begin{align}
		\|\p_t f(t)\|_{H^1} &\les \bra{t}^{-1+\zeta+\de} \ve_1^3,\label{eq:derivative-f}\\
		\| \p_txf(t)\|_{L^2}&\les \bra{t}^{-1+\zeta+2\de}\ve_1^3. \label{eq:derivative-xf}
	\end{align}
\end{lemma}
\begin{proof}
By direct calculation, we have
\begin{align*}
	\p_t f = \p_t( e^{-it\p_x^2}u) = e^{-it\p_x^2} ( \p_t -i \p_x^2) u = e^{-it\p_x^2}( \al \p_x [|u|^2u] + \beta \p_x [\cH(|u|^2)u]).
\end{align*}	
By Lemma~\ref{lem:esti-null} and Lemma~\ref{lem:HLinfty}, we have
\begin{align*}
	\normo{\p_t f(t)}_{H^1} &\les \| \p_x(|u(t)|^2u(t))\|_{H^1}+\| \p_x\big[ \cH(|u(t)|^2)u(t)\big] \|_{H^1}\\
&\les \| \partial_x(|u(t)|^2)\|_{H^1}\| u(t)\|_{W^{1,\infty}}+\left( \| |u(t)|^2\|_{L^\infty}+\|\cH(|u(t)|^2)\|_{L^\infty}\right) \| \p_xu(t)\|_{H^1} \\
&\les \bra{t}^{-2+2\delta}\ve_1^3+\bra{t}^{-1+\zeta+\delta}\ve_1^3\les \bra{t}^{-1+\zeta+\delta}\ve_1^3.
\end{align*}
Similarly, using
\begin{align*}
\wh{\p_t xf}(t,\xi) &=i\p_\xi \left\{ \frac{1}{2\pi}\iint_{\R^2} e^{2it\eta\sigma} i\xi \left( \al -i\beta \sgn{\eta}\right) \wh{f}(t,\xi-\eta)\ol{\wh{f}(t,\xi-\eta-\sigma)}\wh{f}(t,\xi-\sigma) d\eta d\sigma \right\}\\
	&=-\cF \left[ e^{-it\p_x^2}( \al |u|^2u + \beta \cH(|u|^2)u) \right] (t,\xi )\\
&\quad + \frac{i}{2\pi} \iint_{\R^2} e^{2it\eta\sigma} i\xi \left( \al -i\beta \sgn{\eta}\right) \p_\xi \left(\wh{f}(t,\xi-\eta)\ol{\wh{f}(t,\xi-\eta-\sigma)}\wh{f}(t,\xi-\sigma)  \right) d\eta d\sigma, 
\end{align*}
we see that
\begin{align*}
\|\p_t xf(t) \|_{L^2}&\les \| u(t)\|_{L^\infty}^2\| u(t)\|_{L^2}+\| u(t)\|_{W^{1,\infty}}^2\| xf(t)\|_{H^1}+\| \cH(|u(t)|^2)\|_{L^\infty}\| \p_x(xf(t))\|_{L^2} \\
&\les \bra{t}^{-1+\zeta+2\delta} \ve_1^3.
\end{align*}
This completes the proof.
\end{proof}


\section{Energy estimates}\label{sec:energy}

This section is devoted to the proof of the following proposition.
\begin{prop}\label{prop:energy}
	Let $T>0$ and $0<\ve\leq \ve_1\leq 1$. 
	Suppose that $u\in C([0,T], H^2 \cap H^{1,1})$ is a solution to \eqref{kdnls} with initial data $\phi$ satisfying \eqref{thm:initial} and that $u$ satisfies  \eqref{eq:assumption-apriori}. 
	Then, we have
	\begin{align}
		\|u(t)\|_{H^2} \les \ve + \bra{t}^\de\ve_1^2\label{eq:high-u},
	\end{align}
for $t \in [0,T]$.
\end{prop}

\subsection{Setup}

By a standard approximation argument with the local well-posedness result in \cite{KL-LWP}, we may assume $u\in C([0,T],H^k\cap H^{k-1,1})$ for sufficiently large $k$.
In particular, all the following calculations are justified in the appropriate spaces (see Remark~\ref{rem:justify-energy} below).

 From the uniform $L^2$ bound in Lemma~\ref{lem:L2energy}, we have $\| u(t)\|_{L^2}\les \ve$. 
To prove \eqref{eq:high-u}, it suffices to show
\begin{align}
	\|\p_x^2 u(t)\|_{L^2} &\les  \ve + \bra{t}^\de\ve_1^2 \label{eq:high-u-d3}.
\end{align}


As we mentioned in the Introduction, in order to show \eqref{eq:high-u-d3}, we need the gauge transformations to handle the derivative in nonlinearities. 
Here, the simpler gauge \eqref{def:gauge0} is sufficient.
Using the decomposition 
\[ \p_x^2u=P_{\leq 1}\p_x^2u+e^{\rho_0^-[u]}\big( e^{\rho_0^+[u]}Q_+\p_x^2u\big) +e^{\rho_0^+[u]}\big( e^{\rho_0^-[u]}Q_-\p_x^2u\big) ,\]
the estimate
\[ \| P_{\leq 1}\partial_x^2u(t)\|_{L^2} \lesssim \| u(t)\|_{L^2}\lesssim \ve, \] 
and Lemma~\ref{lem:rho}, we have
\[ \| \partial_x^2u(t)\|_{L^2}\lesssim \ve +\sum_\pm \| e^{\rho_0^\pm [u]}Q_\pm \p_x^2u(t)\|_{L^2}.\]
Hence, we focus on proving 
	\begin{align}
		\| e^{\rho_0^\pm [u]}Q_\pm \p_x^2u(t)\|_{L^2} \les \ve + \bra{t}^\de\ve_1^2\label{eq:high-v},
	\end{align}
for $t \in [0,T]$. 
%
%
A direct calculation with integration by parts yields that
\begin{align}
&\frac{d}{dt}\| e^{\rho_0^\pm [u]}Q_\pm \p_x^2u(t)\|_{L^2}^2 \notag \\
&=\int_{\R}\big( \p_t e^{2\rho_0^\pm[u]}\big) |Q_\pm \p_x^2u|^2\,dx \notag \\
&\quad + 2{\rm Re} \int_{\R}e^{2\rho_0^\pm[u]}Q_\pm \p_x^2 \Big[ i\p_x^2u+\alpha \p_x\big[ |u|^2u\big] +\beta \p_x\big[ \mathcal{H}(|u|^2)u\big] \Big] \ol{Q_\pm \p_x^2u}\,dx \notag \\
&\begin{aligned}
&=\int_{\R}(2\p_t\rho_0^\pm[u])\big| e^{\rho_0^\pm[u]}Q_\pm\p_x^2u\big| ^2\,dx \\
&\quad +2{\rm Re} \int_{\R}e^{2\rho_0^\pm[u]}Q_\pm \Big[ \alpha \p_x^3\big[ |u|^2u\big] +\beta \p_x^3\big[ \mathcal{H}(|u|^2)u\big]  \pm i\beta |u|^2Q_\pm \p_x^3u \Big] \ol{Q_\pm \p_x^2u}\,dx .
\end{aligned} \label{eq:energy-v-time}
\end{align}
\begin{rem}\label{rem:justify-energy}
If $u\in C([0,T],H^k)$ for sufficiently large $k$, then $\p_tu\in C([0,T];H^{k-2})$ from the equation \eqref{kdnls}, and one can show that $e^{\rho_0^\pm [u]}Q_\pm \p_x^2u\in C^1([0,T],L^2)$, which justifies $t$-differentiation of its $L^2$ norm in the above formal calculation.
Moreover, functions in $x$-integrals are all smooth and vanishing at infinity, so that integration by parts operations can also be justified.
\end{rem}
By Lemma~\ref{lem:rho}, the first term of \eqref{eq:energy-v-time} is estimated as 
\begin{gather}\label{eq:hi-0a}
\Big| \int_{\R}(2\p_t\rho_0^\pm[u])\big| e^{\rho_0^\pm[u]}Q_\pm\p_x^2u(t)\big| ^2\,dx \Big| 
\les \bra{t}^{-1}\ve_1^2\| e^{\rho_0^\pm[u]}Q_\pm\p_x^2u(t)\|_{L^2}^2 
\les \bra{t}^{-1+2\delta}\ve_1^4.
\end{gather}
The second term of \eqref{eq:energy-v-time} is decomposed into
\[ \sum_{j=1}^5 2{\rm Re} \int_{\R}e^{2\rho_0^\pm[u]}\mathcal{S}^\pm_j \,\ol{Q_\pm \p_x^2u}\,dx,\]
where
\begin{align*}
\mathcal{S}_1^\pm(t)&:=\big( 2\alpha |u|^2+\beta \mathcal{H}(|u|^2)\big) Q_\pm \p_x^3u,\\
\mathcal{S}_2^\pm(t)&:=\big[ Q_\pm ,2\alpha |u|^2+\beta \mathcal{H}(|u|^2)\big] \p_x^3u,\\
\mathcal{S}_3^\pm(t)&:=Q_\pm \big( \alpha u^2\p_x^3\ol{u} +\beta \mathcal{H}(u\p_x^3\ol{u})u \big) ,\\
\mathcal{S}_4^\pm(t)&:=\beta \Big( Q_\pm \big( \mathcal{H}(\ol{u}\p_x^3u)u\big) \pm i|u|^2Q_\pm \p_x^3u\Big) ,\\
\mathcal{S}_5^\pm(t)&:=Q_\pm \Big[ \alpha \big\{ 3\p_x(|u|^2)\p_x^2u+3\p_x^2(|u|^2)\p_xu+3(\p_x^2u\p_x\ol{u}+\p_xu\p_x^2\ol{u})u\big\} \\
&\qquad\qquad +\beta \big\{ 3\p_x\mathcal{H}(|u|^2)\p_x^2u+3\p_x^2\mathcal{H}(|u|^2)\p_xu+3\mathcal{H}(\p_x^2u\p_x\ol{u}+\p_xu\p_x^2\ol{u})u\big\} \Big] .
\end{align*}

Let us consider each term. 
Regarding $\cS_1^\pm$, by integration by parts, we see that
\begin{align*}
&2{\rm Re} \int_{\R}e^{2\rho_0^\pm[u]}\mathcal{S}^\pm_1\, \ol{Q_\pm \p_x^2u}\,dx \\
&=-\int_{\R}\p_x \Big[ e^{2\rho_0^\pm[u]}\big( 2\alpha |u|^2+\beta \mathcal{H}(|u|^2)\big) \Big] |Q_\pm\p_x^2u|^2\,dx\\
&=-\int_{\R}e^{2\rho_0^\pm[u]}\Big[ {\mp} \beta |u|^2\big( 2\alpha |u|^2+\beta \mathcal{H}(|u|^2)\big) + \p_x\big( 2\alpha |u|^2+\beta \mathcal{H}(|u|^2)\big) \Big] |Q_\pm\p_x^2u|^2\,dx.
\end{align*}
By \eqref{eq:linear}, Lemmas~\ref{lem:esti-null}, \ref{lem:HLinfty} and \ref{lem:rho}, we have
\begin{align*}
\Big| 2{\rm Re} \int_{\R}e^{2\rho_0^\pm[u]}\mathcal{S}^\pm_1(t)\, \ol{Q_\pm \p_x^2u(t)}\,dx \Big| 
&\les \Big( \bra{t}^{-2+\zeta}\ve_1^4+\bra{t}^{-\frac32+2\delta}\ve_1^2\Big) \| Q_\pm \p_x^2u(t)\|_{L^2}^2 
\les \bra{t}^{-\frac32+4\delta}\ve_1^4.
\end{align*}
The bound for $\cS_5^\pm$ follows similarly, by \eqref{eq:linear} and Lemma~\ref{lem:esti-null},
\begin{align*}
\| \mathcal{S}_5^\pm (t)\|_{L^2}&\les \| \p_x(|u(t)|^2)\|_{L^\infty}\| \p_x^2u(t)\|_{L^2}+\| \p_x^2(|u(t)|^2)\|_{L^2}\| \p_xu(t)\|_{L^\infty}+\| u(t)\|_{W^{1,\infty}}^2\| \p_x^2u(t)\|_{L^2} \\
&\les \bra{t}^{-1}\ve_1^2\| u(t)\|_{H^2}\les \bra{t}^{-1+\delta}\ve_1^3. 
\end{align*}
Regarding the commutator term $\cS_2^\pm$, we divide it into
\begin{align*}
\| \cS_2^\pm(t)\|_{L^2}&\leq \left\| \p_x^2\big(\big[ Q_\pm, 2\alpha |u|^2+\beta \mathcal{H}(|u|^2)\big] \big) \p_xu\right\|_{L^2} \\
&\quad +2\left\| \big[ Q_\pm, \p_x\big( 2\alpha |u|^2+\beta \mathcal{H}(|u|^2)\big) \big] \p_x^2u\right\|_{L^2} +\left\| \big[ Q_\pm, \p_x^2\big( 2\alpha |u|^2+\beta \mathcal{H}(|u|^2)\big) \big] \p_xu\right\|_{L^2}.
\end{align*}
The first term can be estimated in the same approach as in the proof of Lemma~\ref{lem:R5}, using Lemma~\ref{lem:commutator}.
The other two terms can be estimated directly, by H\"older's inequality and Lemma~\ref{lem:esti-null}.
Indeed, we have
\begin{gather}\label{eq:hi-0b}
\| \cS_2^\pm(t)\|_{L^2} \les \left\| \p_x\big( 2\alpha |u(t)|^2+\beta \mathcal{H}(|u(t)|^2)\big) \right\|_{H^1}\| \p_xu(t)\|_{H^1}\les \bra{t}^{-\frac32+3\delta}\ve_1^3.
\end{gather}
Combining these estimates with Lemma~\ref{lem:rho}, we obtain 
\[ \Big| 2{\rm Re} \int_{\R}e^{2\rho_0^\pm[u]}\big( \mathcal{S}^\pm_1(t)+\cS_2^\pm(t)+\cS_5^\pm(t)\big)  \ol{Q_\pm \p_x^2u(t)}\,dx \Big| \les \bra{t}^{-1+2\delta}\ve_1^4.\]

It remains to handle $\cS_3^\pm$ and $\cS_4^\pm$. 
To this end, we introduce the following lemma.
\begin{lemma}\label{lem:high-energy}
Under the hypotheses of Proposition~\ref{prop:energy}, writing $w:=e^{2\rho_0^\pm[u]}Q_{\pm}\p_x^2u$, we have
\begin{align}
\left| \int_0^t \int_{\R} Q_\pm \big( \alpha [u(s)]^2\ol{\p_x^3 u(s)} +\beta \mathcal{H}(u(s)\ol{\p_x^3 u(s)})u(s) \big)  \ol{w(s)}\,dx\,ds \right| &\les \bra{t}^{2\delta}\ve_1^4, \label{eq:hi-1a}	\\
\left\| Q_\pm \big( u(t)\mathcal{H}(\ol{u(t)}\p_x^3u(t))\big) \pm iu(t)\ol{u(t)}Q_\pm \p_x^3u(t) \right\|_{L^2} &\les \bra{t}^{-1+\delta}\ve_1^3. \label{eq:hi-1b}
\end{align}
\end{lemma}
Putting \eqref{eq:energy-v-time}, \eqref{eq:hi-0a}, \eqref{eq:hi-0b}, and Lemma~\ref{lem:high-energy} together, we get
\[ \| e^{\rho_0^\pm[u(t)]}Q_\pm \p_x^2u(t)\|_{L^2}^2\les \| e^{\rho_0^\pm[\phi]}Q_\pm \p_x^2\phi \|_{L^2}^2+\bra{t}^{2\delta}\ve_1^4\les \ve^2+\bra{t}^{2\delta}\ve_1^4,\]
which shows \eqref{eq:high-v} and completes the proof of Proposition~\ref{prop:energy}.
\hfill $\qed$

\medskip
We devote ourselves to proving Lemma~\ref{lem:high-energy} in the rest of this section.

\subsection{Proof of \eqref{eq:hi-1a}}
We observe that the left-hand side of \eqref{eq:hi-1a} can be written as
\begin{align}\label{eq:hi-1a-f}
	 \int_0^t\int_{\R^3}e^{2is\eta\sigma}(\xi-\eta-\sigma)^3m_1(\xi,\sigma)\wh{f}(s,\xi-\eta)\ol{\wh{f}(s,\xi-\eta-\sigma)}\wh{f}(s,\xi-\sigma)\ol{\wh{h}(s,\xi)}\,d\eta d\sigma d\xi \,ds,
\end{align}
where
\[ m_1(\xi,\sigma):=\frac{(-i)^3}{2\pi}(\varrho_{>1}\chi_\pm) (\xi) \big( \al - i \beta \sgn{\sigma}\big) ,\quad f(t):=e^{-it\p_x^2}u(t),\quad h(t):=e^{-it\p_x^2}w(t).\]
We further decompose \eqref{eq:hi-1a-f} according to the identity
\begin{align*}
(\xi-\eta-\sigma)^2=(\xi-\eta)(\xi-\eta -\sigma) +(\xi-\sigma)(\xi-\eta-\sigma)-(\xi-\eta)(\xi-\sigma) +\eta\sigma  .
\end{align*}
The contribution from the first three terms on the right-hand side 
can be treated straightforwardly, bounded by
\begin{align*}
\int_0^t \| u(s)\|_{W^{1,\infty}}^2\| u(s)\|_{H^2}\| w(s)\|_{L^2}\,ds &\les \int_0^t\bra{s}^{-1+2\delta}\ve_1^4\,ds \les \bra{t}^{2\delta}\ve_1^4,
\end{align*}
Hence, it suffices to consider the term
\begin{gather}\label{eq:hi-1a-f2}
\int_0^t\int_{\R^3}e^{2is\eta\sigma}\eta\sigma(\xi-\eta-\sigma)m_1(\xi,\sigma)\wh{f}(s,\xi-\eta)\ol{\wh{f}(s,\xi-\eta-\sigma)}\wh{f}(s,\xi-\sigma)\ol{\wh{h}(s,\xi)}\,d\eta d\sigma d\xi \,ds.
\end{gather}
Using integration by parts in time, \eqref{eq:hi-1a-f2} can be bounded by the following
\begin{align}
&\begin{aligned}
& \int_{\R^3}e^{2is\eta\sigma}(\xi-\eta-\sigma)m_1(\xi,\sigma)\wh{f}(s,\xi-\eta)\ol{\wh{f}(s,\xi-\eta-\sigma)}\wh{f}(s,\xi-\sigma)\ol{\wh{h}(s,\xi)}\,d\eta d\sigma d\xi  \bigg|_{s=0,t},
\end{aligned} \label{eq:nf-t-1} \\
&\int_0^t \!\int_{\R^3}e^{2is\eta\sigma}(\xi-\eta-\sigma)m_1(\xi,\sigma)\p_s\Big[ \wh{f}(s,\xi-\eta)\ol{\wh{f}(s,\xi-\eta-\sigma)}\wh{f}(s,\xi-\sigma)\Big] \ol{\wh{h}(s,\xi)}\,d\eta d\sigma d\xi \,ds, \label{eq:nf-t-2}\\
& \int_0^t\!\int_{\R^3}e^{2is\eta\sigma}(\xi-\eta-\sigma)m_1(\xi,\sigma)\wh{f}(s,\xi-\eta)\ol{\wh{f}(s,\xi-\eta-\sigma)}\wh{f}(s,\xi-\sigma)\ol{\p_s\wh{h}(s,\xi)}\,d\eta d\sigma d\xi \,ds.\label{eq:nf-t-3}
\end{align}
Then \eqref{eq:nf-t-1} can be bounded by
\[ \max_{s=0,t} \| u(s)\|_{L^\infty}^2\| \p_xu(s)\|_{L^2}\| w(s)\|_{L^2} \les \ve^4+ \bra{t}^{-1+2\de}\ve_1^4 \les \bra{t}^{2\de} \ve_1^4 .\]
Analogously, we estimate \eqref{eq:nf-t-2}, using \eqref{eq:derivative-f} as follows:
\[ \int_0^t \| u(s)\|_{L^\infty}\| u(s)\|_{W^{1,\infty}}\| \p_tf(s)\|_{H^1}\| w(s)\|_{L^2} \,ds\les \int_0^t \bra{s}^{-2+\zeta +2\de}\ve_1^6\,ds\les \ve_1^6
. \]
For \eqref{eq:nf-t-3}, we evaluate this term by
\[ \int_0^t \left\| Q_{\pm} [\al u^2\p_x\ol{u}+ \beta \cH(u\p_x\ol{u}) u] (s)\right\|_{H^1}\| \p_s h(s)\|_{H^{-1}}\,ds. \]
On  one hand, we have
\begin{align*}
	 \left\| Q_{\pm} [\al u^2\p_x\ol{u}+ \beta \cH(u\p_x\ol{u}) u] (s)\right\|_{H^1} &\les \| u(s)\|_{W^{1,\infty}}^2\| \p_xu(s)\|_{H^1}\les \bra{s}^{-1+\delta}\ve_1^3.
\end{align*}
On the other hand, we see that
\begin{align*}
e^{is\p_x^2}\p_s  h&=(\p_s-i\p_x^2)(e^{2\rho_0^\pm[u]}Q_{\pm}\p_x^2u) \\
&=\big[ (\p_s-i\p_x^2)e^{2\rho_0^\pm[u]}\big] Q_{\pm}\p_x^2u -2i\big[ \p_xe^{2\rho_0^\pm[u]}\big] Q_{\pm}\p_x^3u+e^{2\rho_0^\pm[u]}Q_{\pm}\p_x^2e^{is\p_x^2}\p_sf.
\end{align*}
Using the embedding $L^1\hookrightarrow H^{-1}$ and the estimate $\| \phi \psi\|_{H^{-1}}\les \| \phi\|_{W^{1,\infty}}\| \psi\|_{H^{-1}}$, we have
\begin{align*}
\| \p_s h\|_{H^{-1}}&\les \| (\p_s-i\p_x^2)e^{2\rho_0^\pm[u]}\|_{L^\infty}\| Q_{\pm}\p_x^2u\|_{L^2}+\| \p_xe^{2\rho_0^\pm[u]}\|_{W^{1,\infty}}\| Q_{\pm}\p_x^3u\|_{H^{-1}}\\
&\quad +\| e^{2\rho_0^\pm[u]}\|_{W^{1,\infty}}\| Q_{\pm}\p_x^2e^{it\p_x^2}\p_sf\|_{H^{-1}}.
\end{align*}
By Lemma~\ref{lem:rho} and \eqref{eq:derivative-f}, this yields that
\[
 \|\p_s h(s)\|_{H^{-1}} \les \bra{s}^{-1+\zeta+\delta}\ve_1^3.
 \]
Combining these estimates, we obtain 
\[ 
\left|\eqref{eq:nf-t-3} \right|\les \int_0^t\bra{s}^{-2+\zeta+2\de}\ve_1^6\,ds \les \ve_1^6. \]
This completes the proof of \eqref{eq:hi-1a}.
\hfill $\qed$

\subsection{Proof of \eqref{eq:hi-1b}}

The left-hand side of \eqref{eq:hi-1b} can be written as
\[ \Big\| \int_{\R^2}(\xi-\sigma)^3m_2(\xi,\eta,\sigma)\wh{u}(t,\xi-\eta)\ol{\wh{u}(t,\xi-\eta-\sigma)}\wh{u}(t,\xi-\sigma)\,d\eta d\sigma \Big\|_{L^2_\xi} \]
with
\begin{align*}
m_2(\xi,\eta,\sigma)&:=\frac1{2\pi}\Big[ (\varrho_{>1}\chi_\pm)(\xi)(-i)\sgn{\eta}\pm i(\varrho_{>1}\chi_\pm)(\xi-\sigma)\Big] \\
&\;=\frac{-i}{2\pi}\Big[ (\varrho_{>1}\chi_\pm)(\xi)\sgn{\eta} -(\varrho_{>1}\chi_\pm)(\xi-\sigma)\sgn{\xi-\sigma}\Big] .
\end{align*}
We decompose the frequencies $|\xi|,|\xi-\eta|,|\xi-\eta-\sigma|$, and $|\xi-\sigma|$ into the dyadic numbers $N_0,N_1,N_2,$ and $N_3\ge 1$, respectively%
\footnote{We use the inhomogeneous Littlewood-Paley decomposition here.
Thus, we regard $\varrho_1$ as $\varrho_{\leq 1}$ in \eqref{eq:multiplier-hi} and $P_1$ as $P_{\leq 1}$ in the proof below.}%
.
Then we estimate
\begin{align}\label{eq:hi-1b-f}
	 \left\| \int_{\R^2}(\xi-\sigma)^3m_{2,\mathbf{N}}(\xi,\eta,\sigma)\wh{u}(t,\xi-\eta)\ol{\wh{u}(t,\xi-\eta-\sigma)}\wh{u}(t,\xi-\sigma)\,d\eta d\sigma \right\|_{L^2_\xi},
\end{align}
for each 4-tuple $\mathbf{N} = (N_0,N_1,N_2,N_3)$, where
\begin{align}\label{eq:multiplier-hi}
	\begin{aligned}
		m_{2,\mathbf{N}}(\xi,\eta,\sigma)&:= m_2(\xi,\eta,\sigma )\varrho_{\mathbf{N}}(\xi,\eta,\sigma), \\ 
		\varrho_{\mathbf{N}}(\xi,\eta,\sigma) &:= \varrho_{N_0}(\xi)\varrho_{N_1}(\xi-\eta)\varrho_{N_2}(\xi-\eta-\sigma)\varrho_{N_3}(\xi-\sigma).
	\end{aligned}
\end{align}
Then it suffices to consider the following cases:
\begin{align*}
	\begin{aligned}
		\mbox{Case A:}&\quad N_{\max}\sim N_{{\rm med}},\quad \left\{ \begin{aligned} \mbox{Case A1:}&\quad N_1\gtrsim N_3,\\ \mbox{Case A2:}&\quad N_1\ll N_3\sim N_2,\end{aligned}\right. \\
		\mbox{Case B:}&\quad N_0\sim N_1\gg N_2,N_3,\\
		\mbox{Case C:}&\quad N_0\sim N_3\gg N_1,N_2,\\
		\mbox{Case D:}&\quad N_0\sim N_2\gg N_1,N_3,
	\end{aligned}
\end{align*}
where $N_{\max}$ and $N_{{\rm med}}$ are the maximum and the median of $N_1,N_2,N_3$, respectively. 
Note that the high-low interaction (Case C) does not occur thanks to the cancellation property in $m_2$, which is an advantage of the gauge transformation. 
Thus we only consider the remaining cases.

Now, we can bound \eqref{eq:hi-1b-f} roughly, as
\[ \eqref{eq:hi-1b-f}\les \| P_{N_1}u\|_{L^\infty}\| P_{N_2}u\|_{L^\infty}N_3\| P_{N_3}\p_x^2u\|_{L^2}.\]
In Case A1 or Case B, we have $N_1\sim N_{\max}$ and by \eqref{eq:linear-local},
\begin{align*}
\eqref{eq:hi-1b-f}&\les N_1 \| P_{N_1}u\|_{L^\infty}\| P_{N_2}u\|_{L^\infty}\| P_{N_3}\p_x^2u\|_{L^2} \\
&\les N_1^{-\kappa}\bra{t}^{-\frac12}\ve_1\cdot N_2^{-\kappa}\bra{t}^{-\frac12}\ve_1\cdot \bra{t}^{\delta}\ve_1 \les N_{\max}^{-\kappa}\bra{t}^{-1+\delta}\ve_1^3.
\end{align*}
Similarly, in Case A2 or Case D, we have $N_2\sim N_{\max}$ and
\[ \eqref{eq:hi-1b-f}\les \| P_{N_1}u\|_{L^\infty}N_2\| P_{N_2}u\|_{L^\infty}\| P_{N_3}\p_x^2u\|_{L^2} \les N_{\max}^{-\kappa}\bra{t}^{-1+\delta}\ve_1^3. \]
Since $N_0\les N_{\max}$ only contributes, the resulting bound is summable in $\mathbf{N}$ and yields \eqref{eq:hi-1b}.
\hfill $\qed$

\section{Weighted estimates}\label{sec:weighted}

In this section, we prove the following proposition.
\begin{prop}\label{prop:weighted}
Let $T>0$ and $0<\ve\leq \ve_1<1$. 
Suppose that $u\in C([0,T], H^2 \cap H^{1,1})$ is a solution to \eqref{kdnls} with initial data $\phi$ satisfying \eqref{thm:initial} and that $u$ satisfies \eqref{eq:assumption-apriori}. 
Then we have
\begin{align}\label{eq:weighted-u}
\| xf(t)\|_{H^1}\les \ve+\bra{t}^{2\delta}\ve_1^2
\end{align}
for $t\in [0,T]$, where $f(t)=e^{-it\p_x^2}u(t)$.
\end{prop}

\subsection{Setup}

It is convenient to introduce  the operator $J(t):=x+2it\p_x$.
Some of its properties are listed below:
\begin{itemize}
\item $J(t)=e^{it\p_x^2}xe^{-it\p_x^2}$,\quad $J(t)u(t)=e^{it\p_x^2}xf(t)$,\quad $[J(t),\p_t-i\p_x^2]=0$.
\item $J[fg]=2it(\p_xf)g+fJg$,\quad $J[f\ol{g}h]=(Jf)\ol{g}h-f(\ol{Jg})h+f\ol{g}Jh$. \\
$J[f\mathcal{H}(\ol{g}h)]=(Jf)\mathcal{H}(\ol{g}h)-f\mathcal{H}((\ol{Jg})h)+f\mathcal{H}(\ol{g}Jh)$, the same holds if $\cH$ is replaced with $Q_{\pm}$.
\item $[J,\p_x]=-1$,\quad $[J,P_{\leq 1}]=[x,P_{\leq 1}]=(i\varrho'_{\leq 1})(-i\p_x)$,\quad $[J,Q_{\pm}]=[x,Q_{\pm}]=(i\varrho'_{>1}\chi_\pm)(-i\p_x)$. \\
These operators are spatial Fourier multipliers and bounded on $L^p$ for any $1\leq p\leq \infty$.
\item $[J,P_\pm]=[x,P_\pm]:f\mapsto \pm \dfrac{i}{2\pi}\displaystyle\int_{\R}f(x)\,dx$,\quad $[J,\cH]=[x,\cH]:f\mapsto \dfrac{1}{\pi}\displaystyle\int_{\R}f(x)\,dx$. \\
These operators are bounded only from $L^1$ to $L^\infty$.
\end{itemize}
Let us give a proof for the last item.
We see that
\[ x\cH[f](x)-\cH[xf](x) = \frac{1}{\pi} p.v.\int \Big( x\frac{f(y)}{x-y}-\frac{yf(y)}{x-y}\Big) \,dy =\frac{1}{\pi}\int f(y)\,dy,\]
which shows the second claim.
The first claim follows from the identity $P_\pm=\frac12({\rm Id}\pm i\cH)$.

Similarly to the preceding section, we may assume that $u$ is sufficiently regular, so that the following calculations can be justified (however, see Remark~\ref{rem:justify-weighted} below).
We first show the $L^2$ estimate:
\begin{align}\label{eq:weighted-u-0}
\| xf(t)\|_{L^2}=\| J(t)u(t)\|_{L^2}\les \ve+\bra{t}^{2\de}\ve_1^3.
\end{align}
Since $Ju$ satisfies the following relation
\begin{align*}
(\p_t-i\p_x^2)Ju&=\alpha \p_x\Big( 2|u|^2Ju-u^2\ol{Ju}\Big) +\beta \p_x\Big( \cH(|u|^2)Ju+\cH(\ol{u}Ju-u\ol{Ju})u\Big) \\
&\qquad -\alpha |u|^2u-\beta \cH(|u|^2)u,
\end{align*}
by  \eqref{eq:assumption-apriori}, \eqref{eq:linear}, and Lemma~\ref{lem:esti-null}, we obtain
\begin{align*}
\frac{d}{dt}\| J(t)u(t)\|_{L^2}^2&\les \| u(t)\|_{L^\infty}^2\| J(t)u(t)\|_{L^2}\| \p_xJ(t)u(t)\|_{L^2}+\| \p_x\cH(|u(t)|^2)\|_{L^\infty}\| J(t)u(t)\|_{L^2}^2  \\
& \qquad+\| \alpha |u(t)|^2u(t)+\beta \cH(|u(t)|^2)u(t)\|_{L^2}\| J(t)u(t)\|_{L^2}\\
&\les \bra{t}^{-1+2\de}\ve_1^3\| J(t)u(t)\|_{L^2}.
\end{align*}
Dividing by $\| J(t)u(t)\|_{L^2}$ and integrating over $[0,t]$, we have \eqref{eq:weighted-u-0}.
Hence, for \eqref{eq:weighted-u} it suffices to verify
\[ \| \p_xJ(t)u(t)\|_{L^2}\les \ve+\bra{t}^{2\de}\ve_1^2.\]
To overcome the derivative loss, we use the gauge transformations described in Section \ref{sec:gauge}.
By the decomposition
\[ Ju=JP_{\leq 1}u+J(e^{-\rho^+[u]}v_+)+J(e^{-\rho^-[u]}v_-) \]
and the aforementioned properties of $J(t)$, we have
\begin{align*}
&\| \p_xJu\|_{L^2}\\
&\les \| Ju\|_{L^2}+\| u\|_{L^2}\\
&\quad +\sum_{\pm}\Big( t\| \p_xe^{-\rho^\pm[u]}\|_{W^{1,\infty}}\| v_\pm\|_{H^1}+\| \p_xe^{-\rho^\pm[u]}\|_{L^\infty}\| Jv_\pm\|_{L^2}+\| e^{-\rho^\pm[u]}\|_{L^\infty}\| \p_xJv_\pm\|_{L^2}\Big).
\end{align*}
Using \eqref{eq:weighted-u-0}, Lemmas~\ref{lem:L2energy}, \ref{lem:rho}, and Corollary~\ref{cor:apriori-v}, we bound the above term by
\[ \ve+\bra{t}^{2\de}\ve_1^3+ \bra{t}^{\zeta+\de}\ve_1^3+\bra{t}^{-1+\zeta+2\de}\ve_1^3 +\sum_\pm \| \p_xJv_\pm\|_{L^2}\les \ve+\bra{t}^{2\de}\ve_1^3+\sum_\pm \| \p_xJv_\pm\|_{L^2},\]
for $0 < \zeta < \de$.
Therefore, the proof of Proposition~\ref{prop:weighted} is reduced to showing
\begin{align}\label{eq:weighted-v}
\| \p_xJ(t)v_\pm (t)\|_{L^2}\les \ve+\bra{t}^{2\de}\ve_1^2.
\end{align}
From the equation \eqref{eq:v-gauge-transform} for $v_\pm$, one gets
\begin{align*}
&\frac{d}{dt}\| \p_xJ(t)v_\pm(t)\|_{L^2}^2\\
&=2{\rm Re}\int_{\R}\p_x\Big[ J(t)(\p_t-i\p_x^2)v_\pm(t)+i\p_x^2J(t)v_\pm(t)\Big] \ol{\p_xJ(t)v_\pm (t)}\,dx \\
&=2{\rm Re}\int_{\R}\p_xJ(t) \Big[ e^{\rho^\pm[u]}\mathcal{R}_5^\pm +G_\pm [u]v_\pm +i\frac{\beta^2}{2}\wt{\p}_x^{-1}\cH\big[ \cH(|u|^2)\p_x(|u|^2)\big] v_\pm \Big] \ol{\p_xJ(t)v_\pm(t)}\,dx.
\end{align*}
The form of the anti-derivative $\wt{\p}_x^{-1}$ is defined in Definition \ref{def:gauge}. For the last term on the right-hand side, which comes from the second term on the right side of \eqref{eq:v-gauge-transform}, we see that
\begin{gather}\label{eq:antider}
\begin{aligned}
&2{\rm Re}\, i\frac{\beta^2}{2}\int_{\R}\p_xJ\Big[ \wt{\p}_x^{-1}\cH\big[ \cH(|u|^2)\p_x(|u|^2)\big] v_\pm \Big] \ol{\p_xJv_\pm}\,dx\\
&=2{\rm Re}\, i\frac{\beta^2}{2}\int_{\R}\p_x\Big[ 2it\cH\big[ \cH(|u|^2)\p_x(|u|^2)\big] v_\pm \Big] \ol{\p_xJv_\pm}\,dx\\
&\quad +2{\rm Re}\, i\frac{\beta^2}{2}\int_{\R} \cH\big[ \cH(|u|^2)\p_x(|u|^2)\big] Jv_\pm \ol{\p_xJv_\pm}\,dx \\
&\quad +2{\rm Re}\, i\frac{\beta^2}{2}\int_{\R} \wt{\p}_x^{-1}\cH\big[ \cH(|u|^2)\p_x(|u|^2)\big] |\p_xJv_\pm|^2\,dx\\
&=-2t\beta^2{\rm Re}\int_{\R}\p_x\Big[ \cH\big[ \cH(|u|^2)\p_x(|u|^2)\big] v_\pm \Big] \ol{\p_xJv_\pm}\,dx \\
&\quad -\beta ^2{\rm Im}\int_{\R} \cH\big[ \cH(|u|^2)\p_x(|u|^2)\big] Jv_\pm \ol{\p_xJv_\pm}\,dx.
\end{aligned}
\end{gather}
Hence, we get
\begin{align}
&\frac{d}{dt}\| \p_xJ(t)v_\pm(t)\|_{L^2}^2 \nonumber \\
&\quad =2{\rm Re}\int_{\R}\p_xJ \big[ e^{\rho^\pm[u]}\mathcal{R}_5^\pm \big] \ol{\p_xJv_\pm}\,dx \label{eq:wei-1}\\
&\qquad +2{\rm Re}\int_{\R}\p_xJ \big[ G_\pm [u]v_\pm \big] \ol{\p_xJv_\pm}\,dx \label{eq:wei-2}\\
&\qquad -2t\beta^2{\rm Re}\int_{\R}\p_x\Big[ \cH\big[ \cH(|u|^2)\p_x(|u|^2)\big] v_\pm \Big] \ol{\p_xJv_\pm}\,dx \label{eq:wei-3} \\
&\qquad -\beta ^2{\rm Im}\int_{\R} \cH\big[ \cH(|u|^2)\p_x(|u|^2)\big] Jv_\pm \ol{\p_xJv_\pm}\,dx. \label{eq:wei-4}
\end{align}
\begin{rem}\label{rem:justify-weighted}
(i) For the weighted estimates, we need to apply the gauge transformation before applying $\p_xJ$, as $\p_xJe^{\rho^\pm[u]}Q_\pm u$, and not $e^{\rho^\pm[u]}Q_\pm \p_xJu$ as we have done for the $H^2$ energy estimate in the preceding section. Indeed, placed in the $L^2$ norm, the latter would be the same as $e^{\rho_0^\pm[u]}Q_\pm \p_xJu$ and thus become meaningless.

(ii) We remark that the justification of the formal calculations is somewhat more delicate than in the case of the $H^2$ energy estimate.
If $u\in C([0,T],H^k\cap H^{k-1,1})$ for sufficiently large $k$, then $\p_tu\in C([0,T];H^{k-2}\cap H^{k-3,1})$ from the equation \eqref{kdnls}.
However, this does not imply that $\p_xJv_\pm \in C^1([0,T];L^2)$, since $\p_tv_\pm$ has the second term on the right side of \eqref{eq:v-gauge-transform} with possibly unbounded factor $\wt{\p}_x^{-1}\cH\big[ \cH(|u|^2)\p_x(|u|^2)\big]$.

Here, we can justify the derivation of the above differential equality \eqref{eq:wei-1}--\eqref{eq:wei-4} as follows.
We first observe
\begin{align*}
\p_xJv_\pm &=\p_x\Big\{ 2it\p_x\big( e^{\rho^\pm[u]}\big) Q_\pm u+e^{\rho^\pm[u]}JQ_\pm u\Big\} \\
&=e^{\rho^\pm[u]}\Big\{ 2it\big(\palbe (|u|^2)\big)^2Q_\pm u+2it \p_x \big( \palbe(|u|^2)Q_\pm u\big) +\palbe (|u|^2)JQ_\pm u+\p_xJQ_\pm u\Big\} \\
&=:e^{\rho^\pm[u]}F_\pm [u].
\end{align*}
Note that $|\p_xJv_\pm|=|e^{\rho_0^\pm[u]}F_\pm [u]|$, and that $e^{\rho_0^\pm[u]}F_\pm [u]\in C^1([0,T];L^2)$.
Then, we can perform the $t$-differentiation as
\[ \frac{d}{dt}\| \p_xJv_\pm\|_{L^2}^2=\frac{d}{dt}\| e^{\rho_0^\pm[u]}F_\pm[u]\|_{L^2}^2=\int_{\R}\p_t\big( |e^{\rho_0^\pm[u]}F_\pm[u]|^2\big)\,dx =\int_{\R}\p_t\big( |\p_xJv_\pm |^2\big)\,dx.\]
We next calculate $\p_t\big( |\p_xJv_\pm |^2\big)$ pointwise, using the equation \eqref{eq:v-gauge-transform} and noticing that all functions are smooth (while possibly unbounded).
Then, the anti-derivative factor $\wt{\p}_x^{-1}\cH\big[ \cH(|u|^2)\p_x(|u|^2)\big]$ disappears, as we have seen in \eqref{eq:antider}.
Now, remaining functions are all smooth and vanishing at infinity, so we are able to do integration by parts operations.
\end{rem}
We consider the terms on the right-hand side one by one. 
Regarding \eqref{eq:wei-3} and \eqref{eq:wei-4}, by Lemmas~\ref{lem:esti-null}, \ref{lem:HLinfty}, and Corollary~\ref{cor:apriori-v}, we have
\begin{align*}
|\eqref{eq:wei-3}|&\les \bra{t}\| \cH(|u|^2)\|_{W^{1,\infty}}\| \p_x(|u|^2)\|_{H^1}\| v_\pm\|_{H^1}\| \p_xJv_\pm\|_{L^2}\\
&\les \bra{t}\bra{t}^{-1+\zeta}\ve_1^2\bra{t}^{-\frac32+2\delta}\ve_1^2\bra{t}^{\delta}\ve_1\bra{t}^{2\delta}\ve_1\\
&\les \bra{t}^{-\frac32+\zeta +5\delta}\ve_1^6, \\
|\eqref{eq:wei-4}|&\les \|  \cH(|u|^2)\|_{L^\infty} \| \p_x(|u|^2)\|_{L^2} \|Jv_\pm\|_{L^\infty} \|\p_x Jv_\pm\|_{L^2}\\
	&\les \bra{t}^{-1+\zeta}\ve_1^2\bra{t}^{-\frac32+2\de}\ve_1^2\ (\bra{t}^{2\de}\ve_1)^2 \\
&\les \bra{t}^{-\frac52 + \zeta +6\de}\ve_1^6.
\end{align*}
For \eqref{eq:wei-2}, we additionally use Lemma~\ref{lem:G} to obtain 
\begin{align*}
|\eqref{eq:wei-2}|&\les \| \p_xJ(G_\pm [u]v_\pm)\|_{L^2}\| \p_xJv_\pm\|_{L^2} \\
&\les \Big( \bra{t}\| \p_xG_\pm [u]\|_{H^1}\| v_\pm\|_{H^1}+\| G_\pm[u]\|_{W^{1,\infty}}\| Jv_\pm\|_{H^1}\Big) \| Jv_\pm\|_{H^1}\\
&\les \Big( \bra{t}\bra{t}^{-\frac52+\zeta+2\de}\ve_1^4\bra{t}^\delta \ve_1+\bra{t}^{-2+2\delta}\ve_1^4\bra{t}^{2\delta}\ve_1\Big) \bra{t}^{2\de}\ve_1 \les \bra{t}^{-\frac32+\zeta +5\delta}\ve_1^6.
\end{align*}
For the $\mathcal{R}^\pm_5$ part \eqref{eq:wei-1}, we decompose as follows,
\[ \p_xJ(e^{\rho^\pm[u]}\mathcal{R}^\pm _5)=2it \p_x\big( \p_xe^{\rho^\pm [u]}\cdot \mathcal{R}^\pm_5\big) +\p_xe^{\rho^\pm [u]}\cdot J\mathcal{R}^\pm_5 +e^{\rho^\pm [u]}\p_xJ\mathcal{R}^\pm_5.\]
The $L^2$ norm of the first term can be estimated using Lemmas~\ref{lem:rho} and \ref{lem:R5} as
\begin{align*}
\bra{t} \| \p_x\big( \p_xe^{\rho^\pm [u]}\cdot \mathcal{R}^\pm_5\big)\|_{L^2}&\les \bra{t}\| \p_xe^{\rho^\pm [u]}\|_{W^{1,\infty}}\| \mathcal{R}^\pm_5\|_{H^1} \les \bra{t}^{-1+\zeta+\de}\ve_1^5.
\end{align*}
To treat the term $\p_xe^{\rho^\pm [u]}\cdot J\mathcal{R}^\pm_5$, we introduce the following lemma.
\begin{lemma}\label{lem:JR5}
Assume that $u$ satisfies
 \eqref{eq:assumption-apriori}. Then we have
\[ \| J(t)\mathcal{R}^\pm_5(t)\|_{L^2}\les \bra{t}^{-1+\zeta+2\de}\ve_1^3,\]
for any $\zeta >0$.
\end{lemma}
\begin{proof}
Since the terms containing $[J,Q_{\pm}]$ can be treated easier, we may consider as if $J$ commutes with $Q_{\pm}$. Then most of the terms can be estimated by $\| J(t)u(t)\|_{H^1}\|u(t)\|_{W^{1,\infty}}^2\les \bra{t}^{-1+2\de}\ve_1^3$, while we need to bound two terms, $\| \cH(|u|^2)J\p_xu\|_{L^2}$ and $\| \cH(u\ol{\p_xu})Ju\|_{L^2}$.
The former can be treated by Lemma~\ref{lem:HLinfty}.
To estimate the latter, we use H\"older's inequality, Sobolev embedding, and interpolation as follows: For $0<\zeta<1/2$,
\begin{align*}
\| \cH(u\ol{\p_xu})Ju\|_{L^2}&\leq \| \cH(u\ol{\p_xu})\|_{L^{\frac1\zeta}}\| Ju\|_{L^{\frac{2}{1-2\zeta}}}\les \| u\|_{L^2}^{2\zeta}\| u\|_{L^\infty}^{1-2\zeta}\| \p_xu\|_{L^\infty}\| Ju\|_{H^1}\les \bra{t}^{-1+\zeta+2\delta}\ve_1^3.\qedhere
\end{align*}%
\end{proof}
From Lemmas~\ref{lem:rho} and \ref{lem:JR5}, we have
\[ \| \p_xe^{\rho^\pm [u]}\cdot J\mathcal{R}^\pm_5\|_{L^2}\les \bra{t}^{-1+\zeta}\bra{t}^{-1+\zeta+2\de}\ve_1^5\les \bra{t}^{-2+2\zeta+2\de}\ve_1^5,\]
which implies the desired bound.

It remains to estimate
\[2{\rm Re}\int_{\R}e^{\rho^\pm[u]}\p_xJ \mathcal{R}_5^\pm \cdot \ol{\p_xJv_\pm}\,dx.\]
We consider the following decomposition
\begin{align}
\int_{\R} e^{\rho^\pm[u]}\p_xJ\mathcal{R}^\pm_5\cdot \ol{\p_xJv_\pm}\,dx&=-\int_{\R} J\mathcal{R}^\pm_5\cdot \p_x\Big\{ e^{\rho^\pm[u]}\ol{\p_x\big( 2it\p_xe^{\rho^\pm[u]}\cdot Q_{\pm}u\big)}\Big\} \,dx \label{eq:weighted-r1}\\
&\quad -\int_{\R} J\mathcal{R}^\pm_5\cdot \p_x\Big\{ e^{\rho^\pm[u]}\ol{\p_xe^{\rho^\pm[u]}\cdot JQ_{\pm}u}\Big\} \,dx \label{eq:weighted-r2}\\
&\quad +\int_{\R} \p_xJ\mathcal{R}^\pm_5\cdot \ol{e^{2\rho_0^\pm[u]}\p_xJQ_{\pm}u}\,dx. \label{eq:weighted-main}
\end{align} 
Using Lemmas~\ref{lem:rho} and \ref{lem:JR5}, we can estimate the first two terms as
\begin{align*}
\left|\eqref{eq:weighted-r1} \right|&\les \bra{t}\| J\mathcal{R}^\pm_5\|_{L^2}\| e^{\rho^\pm[u]}\|_{W^{1,\infty}}\Big( \| \p_xe^{\rho^\pm[u]}\|_{W^{1,\infty}}\|\p_xQ_{\pm}u\|_{H^1}+\| \p_x^3e^{\rho^\pm[u]}\|_{L^2}\| Q_{\pm}u\|_{L^\infty}\Big) \\
&\les \bra{t}\bra{t}^{-1+\zeta+2\de}\ve_1^3\Big( \bra{t}^{-1+\zeta}\ve_1^2\bra{t}^{\de}\ve_1+\bra{t}^{-\frac32+2\de}\ve_1^2\bra{t}^{-\frac12}\ve_1\Big) \\
&\les \bra{t}^{-1+2\zeta+3\de}\ve_1^6,\\
\left|\eqref{eq:weighted-r2} \right|&\les \| J\mathcal{R}^\pm_5\|_{L^2}\| e^{\rho^\pm[u]}\|_{W^{1,\infty}}\| \p_xe^{\rho^\pm[u]}\|_{W^{1,\infty}} \| JQ_{\pm}u\|_{H^1} \\
&\les \bra{t}^{-1+\zeta+2\de}\ve_1^3\bra{t}^{-1+\zeta}\ve_1^2\Big( \| Ju\|_{H^1}+\| u\|_{H^1}\Big) \\
&\les \bra{t}^{-2+2\zeta+4\delta}\ve_1^6.
\end{align*}

Putting the estimates obtained so far, we have
\[ \frac{d}{dt}\| \p_x J(t)v_\pm(t)\|_{L^2}^2\leq C\bra{t}^{-1+4\de}\ve_1^4 +2{\rm Re}\eqref{eq:weighted-main} ,\]
with $0<\zeta <\delta/2$.
Therefore, the following lemma finishes the proof of \eqref{eq:weighted-v}.
\begin{lemma}\label{lem:high-weighted}
	Let $u$ satisfy  \eqref{eq:assumption-apriori}. Then, writing $w:=e^{2\rho_0^\pm[u]}\p_xJQ_{\pm}u$, we have
	\begin{align}
		\left| \int_0^t\int_{\R} \p_xJ \Big\{ Q_{\pm} (u^2\p_x\ol{u})-2u\p_x\ol{u} Q_{\pm} u\Big\} \ol{w} \,dx \,ds\right| &\les \bra{t}^{4\de} \ve_1^4, \label{eq:we-1} \\
		\left| 2{\rm Re}\int_0^t \!\!\int_{\R} \p_xJ \Big\{ Q_{\pm} \big[ u\mathcal{H}(\p_x\ol{u}\,u)\big] - (Q_{\pm} u)\mathcal{H}(\p_x\ol{u}\,u) -u\p_x\ol{u}Q_{\pm}\mathcal{H}u \Big\} \ol{w} \,dx \,ds\right| &\les \bra{t}^{4\de} \ve_1^4, \label{eq:we-2} \\
		\left| 2{\rm Re}\int_{\R} \p_xJ \Big\{ Q_{\pm} \big[ u\mathcal{H}(\ol{u}\p_xu)\big]  -u\ol{u}Q_{\pm}\mathcal{H} \p_xu\Big\} \ol{w}\,dx\right| &\les \bra{t}^{-1+4\de} \ve_1^4, \label{eq:we-3} \\
		\normo{ \left(\p_xJ [Q_{\pm} ,2\alpha |u|^2 +\beta \mathcal{H}(|u|^2)] \right)\partial_xu }_{L^2} &\les \bra{t}^{-1+2\de} \ve_1^3. \label{eq:we-4}
	\end{align}
\end{lemma}
\begin{rem}\label{rem:gauge}
Here, we mention the reason for using the full gauge transformation $v_\pm$ rather than $e^{\rho_0^\pm[u]}Q_\pm u$.
If we used the simpler one, the term of the form
\[ 2{\rm Re}\int_{\R}\p_xJ\Big\{ \big( 2\alpha |u|^2+\beta \cH(|u|^2)\big) Q_\pm \p_xu\Big\} \ol{w}\,dx \]
would remain.
After distributing the first $\p_xJ$ to the following three functions, this term would create the products of $u\cdot \ol{Ju}\cdot \p_x^2u\cdot \p_x\ol{Ju}$ type.
In the $H^2$ energy estimate in Section~\ref{sec:energy}, the corresponding terms are those of $u\cdot \p_x\ol{u} \cdot \p_x^2u\cdot \p_x^2\ol{u}$ types and can be readily estimated by using the optimal decay of $\| u(t)\|_{L^\infty}$ and $\| \p_xu(t)\|_{L^\infty}$ (recall the estimate for $\cS_5^\pm$).
However, in the weighted energy estimate, these terms become problematic due to the absence of decay effect for $\| J(t)u(t)\|_{L^\infty}$.
\end{rem}

The rest of this section is devoted to the proof of Lemma \ref{lem:high-weighted}.
The strategy is similar to the energy estimates in Lemma~\ref{lem:high-energy}, though  new ideas are needed in some parts due to the lack of symmetry caused by the operator $J$.
We also keep considering as if $J$ would commute with $Q_{\pm}$.

\subsection{Proof of \eqref{eq:we-1}}
After distributing $J$, the left-hand side of \eqref{eq:we-1} can be written as a sum of such terms as 
\[ \int_0^t\int_{\R^3}e^{2is\eta\sigma}\xi(\xi-\eta-\sigma)m_1(\xi,\eta,\sigma)\wh{f}(s,\xi-\eta)\ol{\wh{f}(s,\xi-\eta-\sigma)}\wh{f}(s,\xi-\sigma)\ol{\wh{h}(s,\xi)}\,d\eta d\sigma d\xi \,ds,\]
with one of $f$ being replaced by $xf$
, and with 
\[ m_1(\xi,\eta,\sigma)=\frac1{2\pi}\Big[ (\varrho_{>1}\chi_\pm)(\xi)-2(\varrho_{>1}\chi_\pm)(\xi-\sigma)\Big] ,\quad f(t)=e^{-it\p_x^2}u(t),\quad h(t)=e^{-it\p_x^2}w(t).\]
 Then, from the identity
\[ \xi(\xi-\eta-\sigma) = (\xi-\eta)(\xi-\sigma) -\eta\sigma , \] 
it suffices to consider
\begin{gather}
\int_0^t \int_{\R^3}e^{2is\eta\sigma}m_1(\xi,\eta,\sigma)\wh{\p_xf}(s,\xi-\eta)\ol{\wh{f}(s,\xi-\eta-\sigma)}\wh{\p_xf}(s,\xi-\sigma)  \ol{ \wh{h}(s,\xi) } \,d\eta d\sigma d\xi ds , \label{eq:we-1-1}\\
\int_0^t\int_{\R^3}e^{2is\eta\sigma}\eta\sigma m_1(\xi,\eta,\sigma)\wh{f}(s,\xi-\eta)\ol{\wh{f}(s,\xi-\eta-\sigma)}\wh{f}(s,\xi-\sigma)\ol{\wh{h}(s,\xi)}\,d\eta d\sigma d\xi \,ds, \label{eq:we-1-2}
\end{gather}
with one of $f$ being replaced by $xf$.
Then \eqref{eq:we-1-1} can be easily estimated by
\[ \int_0^t\| u(s)\|_{W^{1,\infty}}^2\| xf(s)\|_{H^1}\| w(s)\|_{L^2}\,ds \les \int_0^t \bra{s}^{-1+4\de}\ve_1^4\,ds\les \bra{t}^{4\de}\ve_1^4.\]
To estimate \eqref{eq:we-1-2}, we perform the integration by parts with respect to time and then obtain the following terms:
\begin{align}
&\begin{aligned}
& \int_{\R^3}e^{2is\eta\sigma}m_1(\xi,\eta,\sigma)\wh{f}(s,\xi-\eta)\ol{\wh{f}(s,\xi-\eta-\sigma)}\wh{f}(s,\xi-\sigma)\ol{\wh{h}(s,\xi)}\,d\eta d\sigma d\xi \bigg|_{s=0,t},
\end{aligned} \label{eq:nf-w-1} \\
&\int_0^t\int_{\R^3}e^{2is\eta\sigma}m_1(\xi,\eta,\sigma)\p_s\Big( \wh{f}(s,\xi-\eta)\ol{\wh{f}(s,\xi-\eta-\sigma)}\wh{f}(s,\xi-\sigma)\Big) \ol{\wh{h}(s,\xi)}\,d\eta d\sigma d\xi \,ds, \label{eq:nf-w-2}\\
&\int_0^t\int_{\R^3}e^{2is\eta\sigma}m_1(\xi,\eta,\sigma)\wh{f}(s,\xi-\eta)\ol{\wh{f}(s,\xi-\eta-\sigma)}\wh{f}(s,\xi-\sigma)\ol{\p_s\wh{h}(s,\xi)}\,d\eta d\sigma d\xi \,ds, \label{eq:nf-w-3}
\end{align}
with one of $f$ being replaced by $xf$.
We also estimate \eqref{eq:nf-w-1} as
\[ \max_{s=0,t}\| u(s)\|_{L^\infty}^2\| xf(s)\|_{L^2}\| w(s)\|_{L^2}\les \ve^4+\bra{t}^{-1+4\de}\ve_1^4\les \bra{t}^{4\delta}\ve_1^4.\]
The estimate of \eqref{eq:nf-w-2} can be done similarly, using \eqref{eq:derivative-f} or \eqref{eq:derivative-xf}, by
\begin{align*}
&\int_0^t \Big( \| u(s)\|_{L^\infty}^2\| \p_sxf(s)\|_{L^2}+\| u(s)\|_{L^\infty}\| e^{is\p_x^2}\p_sf(s)\|_{L^\infty}\| xf(s)\|_{L^2}\Big) \| w(s)\|_{L^2} \,ds\\
&\les \int_0^t \Big( \bra{s}^{-2+\zeta +4\de}+\bra{s}^{-\frac32+\zeta+5\de}\Big) \ve_1^6\,ds\les \ve_1^6.
\end{align*}
For \eqref{eq:nf-w-3}, similarly to \eqref{eq:nf-t-3}, we first estimate it by
\[ \int_0^t \left\| \left[ Q_{\pm} (|u|^2u)-2|u|^2Q_{\pm} u\right](s)\right\|_{H^1}\| \p_sh(s)\|_{H^{-1}}\,ds, \]
with one of $u$ being replaced by $Ju$.
On the one hand, the first norm is bounded by
\[ \| u(s)\|_{W^{1,\infty}}^2\| J(s)u(s)\|_{H^1}\les \bra{s}^{-1+2\delta}\ve_1^3.\]
On the other hand, we see that
\begin{align*}
e^{it\p_x^2}\p_th(t)&=(\p_t-i\p_x^2)(e^{2\rho_0^\pm[u]}\p_xJQ_{\pm}u) \\
&=\big[ (\p_t-i\p_x^2)e^{2\rho_0^\pm[u]}\big] \p_xJQ_{\pm}u -2i\big[ \p_xe^{2\rho_0^\pm[u]}\big] \p_x^2JQ_{\pm}u+e^{2\rho_0^\pm[u]}Q_{\pm}\p_x e^{it\p_x^2}\p_txf(t),
\end{align*}
up to harmless commutator terms.  
Using the embedding $L^1\hookrightarrow H^{-1}$, the estimate $\| \phi \psi\|_{H^{-1}}\les \| \phi\|_{W^{1,\infty}}\| \psi\|_{H^{-1}}$, Lemma~\ref{lem:rho}, and \eqref{eq:derivative-xf}, we have
\begin{align*}
\| \p_sh(s)\|_{H^{-1}}&\les \| (\p_t-i\p_x^2)e^{2\rho_0^\pm[u]}\|_{L^\infty}\| \p_xJQ_{\pm}u\|_{L^2}+\| \p_xe^{2\rho_0^\pm[u]}\|_{W^{1,\infty}}\| \p_x^2JQ_{\pm}u\|_{H^{-1}}\\
&\quad +\| e^{2\rho_0^\pm[u]}\|_{W^{1,\infty}}\| Q_{\pm}\p_x e^{is\p_x^2}\p_sxf\|_{H^{-1}}\\
&\les \bra{s}^{-1+2\de}\ve_1^3+\bra{s}^{-1+\zeta+2\delta}\ve_1^3+\bra{s}^{-1+\zeta +2\delta}\ve_1^3\les \bra{s}^{-1+\zeta +2\delta}\ve_1^3.
\end{align*}
Combining these estimates, we obtain 
\[ |\eqref{eq:nf-w-3}|\les \int_0^t\bra{s}^{-2+\zeta +4\de}\ve_1^6\,ds \les \ve_1^6. \]
This completes the proof of \eqref{eq:we-1}.
\hfill \qed

\subsection{Proof of \eqref{eq:we-2}, when $J$ falls on the second or third $u$}
After distributing $J$, the left-hand side of \eqref{eq:we-2} can be written as a sum of 
\[ \int_0^t\int_{\R^3}e^{2is\eta\sigma}\xi(\xi-\eta-\sigma)m_2(\xi,\eta,\sigma)\wh{f}(s,\xi-\eta)\ol{\wh{f}(s,\xi-\eta-\sigma)}\wh{f}(s,\xi-\sigma)\ol{\wh{h}(s,\xi)}\,d\eta d\sigma d\xi \,ds,\]
with either the second or third $f$ being replaced by $xf$ and 
\[ m_2(\xi,\eta,\sigma)=\frac1{2\pi}\Big[ (\varrho_{>1}\chi_\pm)(\xi)\sgn{\eta}-(\varrho_{>1}\chi_\pm)(\xi-\eta)\sgn{\eta}-(\varrho_{>1}\chi_\pm)(\xi-\sigma)\sgn{\xi-\sigma}\Big] .\]
We make an inhomogeneous dyadic decomposition of $|\xi|,|\xi-\eta|,|\xi-\eta-\sigma|,|\xi-\sigma|$ into $N_0,N_1,N_2,N_3\ge 1$, respectively, and estimate 
\begin{align}\label{eq:we-2-freq}
&\left| \int_0^t\int_{\R^3}e^{2is\eta\sigma}\xi(\xi-\eta-\sigma)m_{2,\mathbf{N}}(\xi,\eta,\sigma)\wh{f}(s,\xi-\eta)\ol{\wh{f}(s,\xi-\eta-\sigma)}\wh{f}(s,\xi-\sigma)\ol{\wh{h}(s,\xi)}\,d\eta d\sigma d\xi \,ds\right| 
\end{align}
for each $N_j \geq 1$ ($j=0, 1,2, 3$), with either the second or third $f$ being replaced by $xf$ and $m_{2,\mathbf{N}}$ defined by \eqref{eq:multiplier-hi}.
As before, we consider the following four cases:
\begin{align*}
\mbox{Case A:}&\quad N_{\max}\sim N_{{\rm med}},\\
\mbox{Case B:}&\quad N_0\sim N_1\gg N_2,N_3,\\
\mbox{Case C:}&\quad N_0\sim N_3\gg N_1,N_2,\\
\mbox{Case D:}&\quad N_0\sim N_2\gg N_1,N_3.
\end{align*}

In Case A, we notice that the term of the form $\| u\mathcal{H}(\ol{Ju}u)]\|_{L^2}$ or $\| u\mathcal{H}(\ol{u}Ju)]\|_{L^2}$ can be estimated by $\| Ju\|_{L^2}\| u\|_{L^\infty}^2$.
Then, by \eqref{eq:linear-local}, we estimate
\begin{align*}
|\eqref{eq:we-2-freq}|& \les \int_0^t N_{\max}^2\| P_{N_1}u\|_{L^\infty}\Big( \| P_{N_2}Ju\|_{L^2}\| P_{N_3}u\|_{L^\infty}+\| P_{N_2}u\|_{L^\infty}\| P_{N_3}Ju\|_{L^2}\Big) \| P_{N_0}w\|_{L^2}\,ds\\
&\les N_{\max}^2(N_1N_2N_3)^{-1-\kappa}(N_2^\kappa +N_3^\kappa) \int_0^t \bra{s}^{-1+4\delta}\ve_1^4\,ds \les N_{\max}^{-\kappa}\bra{t}^{4\de}\ve_1^4,
\end{align*}
which is summable in $\mathbf{N}$.
In Case B, since the first two terms of $m_{2,\bf N}(\xi,\eta,\sigma)$ cancel out, we have
\begin{align*}
\left| \eqref{eq:we-2-freq} \right|&\les \int_0^t N_0N_2\| P_{N_1}u\|_{L^\infty}\Big( \| P_{N_2}Ju\|_{L^2}\| P_{N_3}u\|_{L^\infty}+\| P_{N_2}u\|_{L^\infty}\| P_{N_3}Ju\|_{L^2}\Big) \| P_{N_0}w\|_{L^2}\,ds\\
&\les N_1^{-\kappa}\int_0^t \bra{s}^{-1+4\delta}\ve_1^4\,ds \les N_1^{-\kappa}\bra{t}^{4\de}\ve_1^4,
\end{align*}
which is summable.
Similarly, in Case C the first and the third terms of $m_{2}(\xi,\eta,\sigma)$ cancel out, and we have
\begin{align*}
\left| \eqref{eq:we-2-freq} \right|&\les \int_0^t N_0N_2\| P_{N_1}u\|_{L^\infty}\Big( \| P_{N_2}Ju\|_{L^2}\| P_{N_3}u\|_{L^\infty}+\| P_{N_2}u\|_{L^\infty}\| P_{N_3}Ju\|_{L^2}\Big) \| P_{N_0}w\|_{L^2}\,ds\\
&\les N_3^{-\kappa}\int_0^t \bra{s}^{-1+4\delta}\ve_1^4\,ds +N_1^{-1-\kappa}N_2^{-\kappa}\int_0^t\bra{s}^{-1}\ve_1^2\| P_{N_3}Ju\|_{H^1}\| P_{N_0}w\|_{L^2}\,ds.
\end{align*}
The first term  is summable, and so is the second since $N_3\sim N_0$ in Case C and
\[ \sum_{N_3\sim N_0}\| P_{N_3}J(s)u(s)\|_{H^1}\| P_{N_0}w(s)\|_{L^2}\les \|J(s)u(s)\|_{H^1}\| w(s)\|_{L^2}\les \bra{s}^{4\delta}\ve_1^2\]
by the Cauchy-Schwarz inequality.
From each of them, we obtain the bound $\bra{t}^{4\de}\ve_1^4$.
In Case D, we use the time non-resonance as in the estimate of \eqref{eq:we-1}, having the terms analogous to \eqref{eq:we-1-1} and \eqref{eq:nf-w-1}--\eqref{eq:nf-w-3}.
The bound for the case \eqref{eq:we-1-1} is estimated as
\[ \int_0^t \| \p_xP_{N_1}u\|_{L^\infty}\Big( \| P_{N_2}Ju\|_{L^2}\| \p_xP_{N_3}u\|_{L^\infty}+\| P_{N_2}u\|_{L^\infty}\| \p_xP_{N_3}Ju\|_{L^2}\Big) \| P_{N_0}w\|_{L^2}\,ds \les N_2^{-1}\bra{t}^{4\de}\ve_1^4,\]
which is summable.
The term corresponding to \eqref{eq:nf-w-1} is estimated in the same approach as before.
The terms corresponding to \eqref{eq:nf-w-2} and \eqref{eq:nf-w-3} can be handled due to faster decay, and we only mention that the resulting localized bounds are summable, using the relation $N_2\sim N_0$ and the Cauchy-Schwarz inequality if $\| \p_tP_{N_2}Ju\|_{L^2}$ or $\| P_{N_2}Ju\|_{H^1}$ appears, while using a negative power of $N_2$ in the other cases.
\hfill \qed

\subsection{Proof of \eqref{eq:we-3}, when $J$ falls on the second or third $u$}
It suffices to consider
\[ \left| \int_{\R^3}e^{2it\eta\sigma}\xi(\xi-\sigma)m_3(\xi,\eta,\sigma)\wh{f}(t,\xi-\eta)\ol{\wh{f}(t,\xi-\eta-\sigma)}\wh{f}(t,\xi-\sigma)\ol{\wh{h}(t,\xi)}\,d\eta d\sigma d\xi \right| \]
with either the second or third $f$ being replaced by $xf$ and 
\[ m_3(\xi,\eta,\sigma)=\frac1{2\pi}\Big[ (\varrho_{>1}\chi_\pm)(\xi)\sgn{\eta}-(\varrho_{>1}\chi_\pm)(\xi-\sigma)\sgn{\xi-\sigma}\Big] .\]
As before, we localize each function and consider the estimate of
\[ \left| \int_{\R^3}e^{2it\eta\sigma}\xi(\xi-\sigma)m_{3,\mathbf{N}}(\xi,\eta,\sigma)\wh{f}(t,\xi-\eta)\ol{\wh{f}(t,\xi-\eta-\sigma)}\wh{f}(t,\xi-\sigma)\ol{\wh{g}(t,\xi)}\,d\eta d\sigma d\xi \right|, \]
and make the same case separation into A, B, C, and D, with Case C vanishing.
The proofs in Cases A and B are the same as the corresponding cases for \eqref{eq:we-2-freq}.
In Case D, the bound is 
\begin{align*}
&N_2N_3\| P_{N_1}u\|_{L^\infty}\Big( \| P_{N_2}Ju\|_{L^2}\| P_{N_3}u\|_{L^\infty}+\| P_{N_2}u\|_{L^\infty}\| P_{N_3}Ju\|_{L^2}\Big) \| P_{N_0}w\|_{L^2}\\
&\les N_1^{-1-\kappa}N_3^{-\kappa}\bra{t}^{-1}\ve_1^2\| P_{N_2}Ju\|_{H^1}\| P_{N_0}w\|_{L^2}+N_2^{-\kappa}\bra{t}^{-1+4\de}\ve_1^4,
\end{align*}
which is summable in $\mathbf{N}$ using the relation $N_2\sim N_0\sim N_{\max}$ and yields $\bra{t}^{-1+4\de}\ve_1^4$.
\hfill \qed

\subsection{Proof of \eqref{eq:we-2} and \eqref{eq:we-3}, when $J$ falls on the first $u$}
To deal with this case, we rearrange these terms in \eqref{eq:we-2}--\eqref{eq:we-3} as
\begin{align*}
&\p_x\Big\{ Q_{\pm} \big[ (Ju)\mathcal{H}(\p_x\ol{u}\,u)\big] - (Q_{\pm} Ju)\mathcal{H}(\p_x\ol{u}\,u) -(Ju)\p_x\ol{u}\,Q_{\pm}\mathcal{H}u \Big\} \\
&\hspace{5cm} + \p_x\Big\{ Q_{\pm} \big[ (Ju)\mathcal{H}(\ol{u}\p_xu)\big]  -(Ju)\ol{u}Q_{\pm}\mathcal{H} \p_xu\Big\} \\
&=\p_xQ_{\pm}\big[ (Ju)\p_x\cH(|u|^2)\big] -\p_x\big[ (Ju)\p_x( \ol{u}Q_{\pm}\cH u\big) \big] - (Q_{\pm} Ju)\p_x \mathcal{H}(\p_x\ol{u}\,u) -(\p_x JQ_{\pm}u)\mathcal{H}(\p_x\ol{u}\,u).
\end{align*}
It suffices to prove 
\begin{gather}
\left\| (Ju)\p_x\cH(|u|^2)\right\|_{H^1}+\left\| (Ju)\p_x( \ol{u}Q_{\pm}\cH u ) \right\|_{H^1}+\left\| (Q_{\pm} Ju)\p_x \mathcal{H}(\p_x\ol{u}\,u)\right\|_{L^2}\les \bra{t}^{-1+2\de}\ve_1^3, \label{eq:we-23-1}\\
\left| 2{\rm Re}\int_{\R}(\p_x JQ_{\pm}u)\mathcal{H}(\p_x\ol{u}\,u)\ol{w}\,dx\right| \les \bra{t}^{-1+4\de}\ve_1^4, \label{eq:we-23-2}
\end{gather}
which verify \eqref{eq:we-2} and \eqref{eq:we-3} in this case.

For \eqref{eq:we-23-1}, we exploit the structure of $\p_x(u\ol{v})$ by Lemma~\ref{lem:esti-null}.
Indeed, the second term can be estimated as
\begin{align*}
\| (Ju)\p_x(\ol{u}Q_{\pm}\cH u)\|_{H^1}&\les \| Ju\|_{H^1}\frac{1}{t}\| \ol{u}JQ_{\pm}\cH u-\ol{Ju}Q_{\pm}\cH u\|_{H^1}\\
&\les \frac{1}{t}\| Ju\|_{H^1}\Big( \| u\|_{W^{1,\infty}}\| JQ_{\pm}\cH u\|_{H^1}+\| Ju\|_{H^1}\| Q_{\pm}\cH u\|_{W^{1,\infty}}\Big) \\
&\les t^{-1}\bra{t}^{-\frac12+4\de}\ve_1^3
\end{align*}
for $t\geq 1$ and
\[ \| (Ju)\p_x(\ol{u}Q_{\pm}\cH u)\|_{H^1}\les \| Ju\|_{H^1}\| u\|_{H^2}^2\les \ve_1^3\]
for $0\leq t\leq 1$.
The other terms in \eqref{eq:we-23-1} can be estimated similarly.
For \eqref{eq:we-23-2} we can derive the same structure by taking the real part, as 
\begin{align*}
\text{LHS of \eqref{eq:we-23-2}}&=\left| 2{\rm Re}\int_{\R}\cH(\p_x\ol{u}\,u)e^{2\rho_0^\pm[u]}|\p_xJQ_{\pm}u|^2\,dx \right| \\
&=\left| \int_{\R}\p_x\cH(|u|^2)e^{2\rho_0^\pm[u]}|\p_xJQ_{\pm}u|^2\,dx \right| \\
&\les \| \p_x\cH(|u|^2)\|_{L^\infty}\| \p_xJQ_{\pm}u\|_{L^2}^2\les \bra{t}^{-\frac32+6\de}\ve_1^4.
\end{align*}
This completes the proof of \eqref{eq:we-2} and \eqref{eq:we-3}.
\hfill \qed

\subsection{Proof of \eqref{eq:we-4}}
We only consider the term $\cH(|u|^2)$ for simplicity.
First, note that
\begin{align}
&\p_xJ\Big( Q_{\pm}\big( \cH(|u|^2)\p_xu\big) -\cH(|u|^2)Q_{\pm}\p_xu\Big) \notag \\
&=\p_x[J,Q_{\pm}]\big( \cH(|u|^2)\p_xu\big) -\p_x\big( \cH(|u|^2)[J,Q_{\pm}]\p_xu\big) \label{eq:we-4-1} \\
&\quad + \left(\p_x[Q_{\pm},\cH((Ju)\ol{u}-u(\ol{Ju}))] \right)\p_xu \label{eq:we-4-2} \\
&\quad + \left(\p_x[Q_{\pm},\cH(|u|^2)] \right)J\p_xu. \label{eq:we-4-3}
\end{align}
Recalling that $[J,Q_{\pm}]$ is a bounded Fourier multiplier, we have, by Lemma \ref{lem:HLinfty},
\[ \| \eqref{eq:we-4-1}\|_{L^2}\les \| \p_x\cH(|u|^2)\|_{L^2}\| \p_xu\|_{L^\infty}+\| \cH(|u|^2)\|_{L^\infty}\| \p_x^2u\|_{L^2}\les \bra{t}^{-1+\zeta+\de}\ve_1^3.\]
For \eqref{eq:we-4-2}--\eqref{eq:we-4-3}, we utilize the commutator estimate in Lemma~\ref{lem:commutator}.
Then, we estimate \eqref{eq:we-4-2} (using the first bound in \eqref{eq:comm-all}) after applying inhomogeneous decomposition on each function, as follows:
\begin{align*}
\| P_{N_0}\eqref{eq:we-4-2}\|_{L^2}&\les \sum_{N_1,N_2,N_3\geq 1}N_0N_3\| \cH((P_{N_1}Ju)\ol{P_{N_2}u})\|_{L^2}\| P_{N_3}u\|_{L^\infty} \\
&\les \sum_{N_1,N_2,N_3\geq 1}N_0N_3N_1^{-1}\| P_{N_1}Ju\|_{H^1}N_2^{-1-\kappa}N_3^{-1-\kappa}\bra{t}^{-1}\ve_1^2.
\end{align*}
Since the commutator vanishes if $N_3\gg N_1,N_2$, we may assume $N_0\les \max\{ N_1,N_2\}$. 
Then this bound can be summed up to give $\bra{t}^{-1+2\de}\ve_1^3$ when $N_1\les N_2$ or $N_2\ll N_1\les N_3$.
If $N_1\gg \max\{ N_2,N_3\}$, we can use the relation $N_1\sim N_0$ which allows us to have the same bound by summing up in $\ell^1_{N_1,N_2,N_3}$ and then in $\ell^2_{N_0}$.
Finally, \eqref{eq:we-4-3} is estimated similarly to Lemma~\ref{lem:commutator} (but in a slightly different manner), as
\[ \| \eqref{eq:we-4-3}\|_{L^2}\les \| \p_x\cH(|u|^2)\|_{H^1}\| J\p_xu\|_{L^2}\les \bra{t}^{-\frac32+4\de}\ve_1^3.\]

This finishes the proof of Lemma~\ref{lem:high-weighted}, and therefore, that of Proposition~\ref{prop:weighted}.
\hfill \qed

\section{Proof of asymptotic behavior}\label{sec:asymptotic}

In this section, we devote ourselves to proving Proposition \ref{prop:linfty} and the asymptotic behavior of the solution to \eqref{kdnls}. To this end, we begin with determining the exact phase correction for \eqref{kdnls}. We continue to use the notation $f(t)=e^{-it\partial_x^2}u(t)$. Taking Fourier transform, the equation for $\wh{f}(t,\xi)$ becomes 
\begin{align*}
\partial_t\wh{f}(t,\xi) &=\frac{i\alpha}{2\pi}\xi \int_{\mathbb{R}^2}e^{2it\eta \sigma}\wh{f}(t,\xi-\eta)\ol{\wh{f}(\xi-\eta-\sigma)}\wh{f}(\xi-\sigma )\,d\eta d\sigma \\
&\quad +\frac{\beta}{2\pi}\xi \int_{\mathbb{R}^2}e^{2it\eta \sigma}\sgn{\eta}\wh{f}(t,\xi-\eta)\ol{\wh{f}(\xi-\eta-\sigma)}\wh{f}(\xi-\sigma )\,d\eta d\sigma .
\end{align*}
On the one hand, by the formula 
\[ \mathcal{F}_\eta [1] (y)=\sqrt{2\pi}\delta (y),\quad \mathcal{F}_\eta [\sgn{\eta}](y)=-i\sqrt{\frac{2}{\pi}} p.v.\frac{1}{y} ,\]
we see that
\begin{align*}
 \mathcal{F}_{\eta ,\sigma}[e^{2it\eta \sigma} ](y,z) &=\mathcal{F}_\sigma \left[ \sqrt{2\pi} \delta (y-2t\sigma )\right] (z) =\mathcal{F}^{-1}_\sigma \left[ \frac{\sqrt{2\pi}}{2t}\delta \Big(\sigma -\frac{y}{2t}\Big) \right] (-z) =\frac{1}{2t}e^{-i\frac{yz}{2t}}, \\
\cF_{\eta,\sigma}\left[ e^{2it\eta \sigma} \sgn{\eta} \right](y,z) &= \cF_\sigma \left[ -i\sqrt{\frac{2}{\pi}} p.v.\frac{1}{y-2t\sigma} \right] (z) = -\frac{1}{2t} \cF_\sigma^{-1} \left[ -i\sqrt{\frac{2}{\pi}} p.v.\frac{1}{\sigma -\frac{y}{2t}} \right] (-z)\\
&=-\frac{1}{2t} e^{i(-z)\frac{y}{2t}}\sgn{-z} =\frac{1}{2t}e^{-i\frac{yz}{2t}}\sgn{z}.
\end{align*}
On the other hand, for general functions $f_1,f_2,f_3$ we use the inverse transform to obtain
\begin{align*}
\mathcal F_\eta^{-1} \left[ \wh{f_1}(\xi-\eta) \ol{\wh{f_2}(\xi-\eta-\sigma)}\right] (y) &= \frac{1}{\sqrt{2\pi}} \int_{\R} e^{iy\eta}\wh{f_1}(\xi -\eta )\frac{1}{\sqrt{2\pi}}\int_{\R} e^{ix(\xi -\eta -\sigma )}\ol{f_2(x)}\,dx \,d\eta \\
	&= \frac{1}{\sqrt{2\pi}}e^{iy\xi}\int_{\R}  e^{-ix\sigma} f_1(x-y)\ol{f_2(x)}dx,
\end{align*}
which yields that
\begin{align*}
	&\mathcal F_{\eta,\sigma}^{-1} \left[\wh{f_1}(s,\xi-\eta) \ol{\wh{f_2}(s,\xi-\eta-\sigma)}\wh{f_3}(s,\xi-\sigma)\right] (y,z)\\
	&=\frac{1}{\sqrt{2\pi}}e^{iy\xi} \int_{\R} f_1(x-y)\ol{f_2(x)} \mathcal F_{\sigma}^{-1} \left[  e^{-ix\sigma}  \wh{f_3}(\xi-\sigma) \right] (z)\,dx\\
	&=\frac{1}{\sqrt{2\pi}}e^{iy\xi} \int_{\R} f_1(x-y)\ol{f_2(x)} e^{i(z-x)\xi} f_3(x-z)\,dx\\
	&=\frac{1}{\sqrt{2\pi}}e^{i(y+z)\xi} \int_{\R} e^{-ix\xi}f_1(x-y)\ol{f_2(x)} f_3(x-z)\,dx .
\end{align*}
From these formulas and Plancherel's theorem, we see that 
\begin{align*}
&\frac{i\alpha}{2\pi}\xi \int_{\R^2} e^{2it\eta \sigma}\wh{f_1}(t,\xi-\eta)\ol{\wh{f_2}(\xi-\eta-\sigma)}\wh{f_3}(\xi-\sigma )\,d\eta d\sigma \\
&= \frac{i\alpha}{2\pi}\xi \int_{\R^2} \mathcal{F}_{\eta,\sigma}\left[ e^{2it\eta \sigma}\right] (y,z) \mathcal{F}^{-1}_{\eta,\sigma}\left[ \wh{f_1}(t,\xi-\eta)\ol{\wh{f_2}(\xi-\eta-\sigma)}\wh{f_3}(\xi-\sigma )\right] (y,z) \,dy dz \\
&=\frac{i\alpha}{2t}(2\pi)^{-\frac32}\xi \int_{\R^3} e^{-i\frac{yz}{2t}}e^{-i\xi (x-y-z)}f_1(x-y)\ol{f_2(x)} f_3(x-z)\,dx dy dz,
\end{align*}
and similarly,
\begin{align*}
&\frac{\beta}{2\pi}\xi \int_{\R^2} e^{2it\eta \sigma}\sgn{\eta}\wh{f_1}(t,\xi-\eta)\ol{\wh{f_2}(\xi-\eta-\sigma)}\wh{f_3}(\xi-\sigma )\,d\eta d\sigma \\
&=\frac{\beta}{2t}(2\pi)^{-\frac32} \xi \int_{\R^3} e^{-i\frac{yz}{2t}}\sgn{z} e^{-i\xi (x-y-z)}f_1(x-y)\ol{f_2(x)} f_3(x-z)\,dx dy dz.
\end{align*}
In principle, the resonance parts of these terms are given by dropping the oscillating factor $e^{-i\frac{yz}{2t}}$, as
\begin{align}\label{eq:correction-alpha}
\begin{aligned}
		&\frac{i\alpha}{2t}(2\pi)^{-\frac32}\xi \int_{\R^3} e^{-i\xi (x-y-z)}f_1(x-y)\ol{f_2(x)} f_3(x-z)\,dx dy dz \\
&=\frac{i\alpha}{2t}(2\pi)^{-\frac32}\xi \int_{\R^3} e^{-i\xi (y+z-x)}f_1(y)\ol{f_2(x)} f_3(z)\,dx dy dz \\
&=\frac{i\alpha}{2t}\xi \wh{f_1}(\xi)\ol{\wh{f_2}(\xi)}\wh{f_3}(\xi ),
\end{aligned}
\end{align}
and similarly,
\begin{align*}
	&\frac{\beta}{2t}(2\pi)^{-\frac32} \xi \int_{\R^3} \sgn{z} e^{-i\xi (x-y-z)}f_1(x-y)\ol{f_2(x)} f_3(x-z)\,dx dy dz\\
&=\frac{i\beta}{2t}\xi \frac{1}{\sqrt{2\pi}}\int_{\R} e^{-i\xi y}f_1(y)\,dy \frac{1}{2\pi}\int_{\R^2} e^{i\xi (x-z)}(-i\sgn{x-z}) \ol{f_2(x)} f_3(z)\,dx dz \\
&=\frac{i\beta}{2t}\xi \wh{f_1}(\xi) \mathcal{H} \left[ \ol{\wh{f_2}}\wh{f_3}\right] (\xi ).
\end{align*}
We thus obtain the resonance part of $\partial_t\wh{f}(t,\xi)$ as
\[ 
i\Big( \frac{\alpha}{2t}\xi |\wh{f}(t, \xi)|^2 + \frac{\beta}{2t}\xi \mathcal{H} \left[ |\wh{f}(t)|^2\right] (\xi )\Big) \wh{f}(t,\xi). \]

Then Proposition \ref{prop:linfty} follows from the following lemma.
\begin{lemma}\label{lem:asymptotic}
	Assume that $u \in C([0,T], H^{2}\cap H^{1,1})$ is a solution to \eqref{kdnls} satisfying  \eqref{eq:assumption-apriori}.
	Then, for $0\leq t\leq T$ we have
	\begin{align}\label{eq:linfty-apriori}
		\left\|\bra{\xi}\wh{f}(t,\xi)\right\|_{L_\xi^\infty}  \les \ve_1.
	\end{align}
	Moreover, we have, for $1\leq t_1 \leq t_2\leq T$,
	\begin{align}\label{eq:asymptotic}
		\left\|\bra{\xi}\left(e^{-iB(t_2,\xi)}\wh{f}(t_2,\xi) - e^{-iB(t_1,\xi)}\wh{f}(t_1,\xi) \right) \right \|_{L_\xi^\infty}  \les t_1^{-\de}\ve_1^3,
	\end{align}
	where $B(t,\xi)$ is the real valued function defined by 
	\begin{align*}
		B(t,\xi) &:= B_1(t,\xi) +  B_2(t,\xi),\\
		B_1(t,\xi)&:= \al \int_1^t \frac1{2s} \xi  \left|\wh{f}(s,\xi)\right|^2 ds,	\;\;\mbox{ and }\;\;	B_2(t,\xi):= \beta \int_1^t \frac1{2s} \xi \mathcal{H}\Big( \left|\wh{f}(s,\cdot )\right|^2\Big) (\xi) ds.
	\end{align*}
	Furthermore, there exists $\kappa>0$ such that for any dyadic $N\geq 1$ and $0\leq t\leq T$ we have
	\begin{equation}\label{eq:linfty-local2}
		\left\| \varrho_N(\xi) \bra{\xi}\wh{f}(t,\xi)\right\|_{L_\xi^\infty}  \les N^{-\kappa}\ve_1.
	\end{equation}
\end{lemma}


As mentioned in Remark \ref{rem:hayashi-naumkin}, Hayashi and Naumkin \cite{HN1997-dnls,HN1998-dnls,hayashi-naumkin1998} first established the corresponding estimates for the NLS and the DNLS. While our proof of Lemma~\ref{lem:asymptotic} is closely related to their argument, it is based on the Fourier analysis (cf.~\cite{kapu}). 

\begin{rem}
By Lemma~\ref{lem:linear} and  \eqref{eq:assumption-apriori}, we have 
\begin{gather}\label{eq:linear'}
\begin{aligned}
\| u(t)\|_{W^{1,\infty}}&\lesssim \bra{t}^{-\frac12}\| \bra{\xi}\wh{f}(t)\|_{L^\infty_\xi}+\bra{t}^{-\frac12-\gamma}\| f(t)\|_{H^{1,1}}\qquad (0<\gamma <\tfrac14) \\
&\lesssim \bra{t}^{-\frac12}\| f(t)\|_{H^{1,1}}\les \bra{t}^{-\frac12+2\delta}\ve_1.
\end{aligned}
\end{gather}
We only use this non-sharp $L^\infty$ decay in the proof of Lemma \ref{lem:asymptotic}. 
Note that the sharp decay \eqref{eq:linear} is a consequence of Lemma~\ref{lem:asymptotic}.
\end{rem}

\begin{proof}[Proof of Lemma \ref{lem:asymptotic}]
Let us define
\begin{align*}
R_1(t,\xi )&:=\frac{i\alpha}{2\pi}\xi \int_{\R^2} e^{2it\eta \sigma}\wh{f}(t,\xi-\eta)\ol{\wh{f}(\xi-\eta-\sigma)}\wh{f}(\xi-\sigma )\,d\eta d\sigma - \frac{i\alpha}{2t}\xi |\wh{f}(t, \xi)|^2\wh{f}(t,\xi ) ,\\
R_2(t,\xi )&:=\frac{\beta}{2\pi}\xi \int_{\R^2} e^{2it\eta \sigma}\sgn{\eta}\wh{f}(t,\xi-\eta)\ol{\wh{f}(\xi-\eta-\sigma)}\wh{f}(\xi-\sigma )\,d\eta d\sigma -\frac{i\beta}{2t}\xi \mathcal{H} \left[ |\wh{f}(t)|^2\right] (\xi )\wh{f}(t,\xi).
\end{align*}
Then we have
 \begin{align}\label{eq:time-deri-f}
 	\partial_t\wh{f}(t,\xi )=i\partial_tB(t,\xi )\wh{f}(t,\xi )+R_1(t,\xi )+R_2(t,\xi),
 \end{align}
which implies 
\[
 e^{-iB(t_2,\xi)}\wh{f}(t_2,\xi ) -e^{-iB(t_1,\xi )}\wh{f}(t_1,\xi )=\int_{t_1}^{t_2}e^{-iB(s,\xi )}\Big\{ R_1(s,\xi)+R_2(s,\xi )\Big\} \,ds 
\]
for any $1\leq t_1\leq t_2\leq T$.
In particular, to prove \eqref{eq:linfty-apriori} and \eqref{eq:asymptotic}, it suffices to show
\begin{align}
\| R_1(t,\xi )\|_{L^\infty_\xi}+\| R_2(t,\xi)\|_{L^\infty_\xi}&\lesssim t^{-1-\de}\ve_1^3,\label{eq:asymptotic0} \\
\| \xi R_1(t,\xi )\|_{L^\infty_\xi}+\| \xi R_2(t,\xi)\|_{L^\infty_\xi}&\lesssim t^{-1-\de}\ve_1^3 \label{eq:asymptotic1}
\end{align}
for $t\geq 1$.
Note that for $0\leq t\leq 1$ we have $\| \bra{\xi}\wh{f}(t,\xi )\|_{L^\infty_\xi}\les \| f(t)\|_{H^{1,1}}\lesssim \ve_1$ by the assumption \eqref{eq:assumption-apriori}.

We begin with the proof of \eqref{eq:asymptotic0}. From the calculation \eqref{eq:correction-alpha} and integration by parts in $x$, we see that
\begin{align*}
&|R_1(t,\xi )|\les t^{-1}\left| \int_{\R^3} \Big( e^{-i\frac{yz}{2t}}-1\Big) \xi e^{-i\xi (x-y-z)}f(t,x-y)\ol{f(t,x)} f(t,x-z)\,dx dy dz\right| \\
&\les t^{-1} \left( \left| \int_{\R^3} \Big( e^{-i\frac{yz}{2t}}-1\Big) e^{-i\xi (x-y-z)}(\partial_xf)(t,x-y)\ol{f(t,x)} f(t,x-z)\,dx dy dz\right| + (\text{symm.~terms})\right) \\
&\les t^{-1-\zeta}\left( \int_{\R^3} |yz|^{\zeta} \left| (\partial_xf)(t,x-y)\ol{f(t,x)} f(t,x-z)\right| \,dx dy dz+ (\text{symm.~terms})\right) \\
&\les t^{-1-\zeta} \| \bra{x}^{2\zeta}\partial_xf(t)\|_{L^1}\| \bra{x}^{2\zeta}f(t)\|_{L^1}\| f(t)\|_{L^1} \\
&\les t^{-1-\zeta}\| f(t)\|_{H^{1,1}}^3 \les t^{-1-\zeta+6\delta}\ve_1^3.
\end{align*}
By setting $\zeta =7\delta \in (0,\frac14)$, we obtain the desired bound. The estimate for $R_2(t,\xi )$ can be obtained in the same way from the fact that the sign function $\sgn{z}$ does not affect the integration by parts in $x$.

To prove \eqref{eq:asymptotic1},  we need refined estimates.
Let us consider 
\begin{align*}
\wt R_1(t,\xi )&:=\xi \Big( R_1(t,\xi )+\frac{i\alpha}{2t}\xi |\wh{f}(t, \xi)|^2\wh{f}(t,\xi ) \Big) \\
&\;=\frac{i\alpha}{2\pi}\int_{\R^2} e^{2it\eta \sigma}\xi^{2}\wh{f}(t,\xi-\eta)\ol{\wh{f}(t,\xi-\eta-\sigma)}\wh{f}(t,\xi-\sigma )\,d\eta d\sigma .
\end{align*}
Decomposing the multiplier $\xi^{2}$ as
\begin{align}\label{eq:algebraic}
	 \xi^{2} = (\xi^{2}-(\xi-\eta)^{2})+(\xi-\eta)\sigma +(\xi-\eta)(\xi -\eta -\sigma) ,
\end{align}
we denote the corresponding parts by $\wt{R}_{1,1}$, $\wt{R}_{1,2}$ and $\wt{R}_{1,3}$, instead of $\xi^2$ in $\wt R_1$, respectively.
The first part $\wt{R}_{1,1}$ consists of the terms of the form
\[ \int_{\R^2} e^{2it\eta \sigma}\eta \wh{\partial_x^{k_1}f}(t,\xi-\eta)\ol{\wh{\partial_x^{k_2}f}(t,\xi-\eta-\sigma)}\wh{\partial_x^{k_3}f}(t,\xi-\sigma )\,d\eta d\sigma,\quad 0\leq k_j\leq 1,~k_1+k_2+k_3=1.\]
By integration by parts in $\sigma$, \eqref{eq:linear'} and the assumption \eqref{eq:assumption-apriori}, we obtain an extra decay $t^{-1}$ and 
\begin{align*}
\| \wt{R}_{1,1}(t,\xi )\|_{L^\infty_\xi}&\les t^{-1}\Big( \| \partial_x^{k_1}u(t)\ol{e^{it\partial_x^2}x\partial_x^{k_2}f(t)}\partial_x^{k_3}u(t)\|_{L^1_x}+\| \partial_x^{k_1}u(t)\ol{\partial_x^{k_2}u(t)}e^{it\partial_x^2}x\partial_x^{k_3}f(t)\|_{L^1_x}\Big) \\
&\les t^{-1}\| u(t)\|_{W^{1,\infty}} \| u(t)\|_{H^1}\| f(t)\| _{H^{1,1}} \\
&\les t^{-\frac32+5\de}\ve_1^3.
\end{align*}
The estimate of $\wt{R}_{1,2}$ is parallel, for which we integrate by parts in $\eta$.
For $\wt{R}_{1,3}$, we set $f_1(t):=-i\partial_xf(t)$, $f_2(t):=-i\partial_xf(t)$, $f_3(t):=f(t)$ and repeat the previous calculation to figure out the main part,
\begin{align*}
\wt{R}_{1,3}(t,\xi )&=\frac{i\alpha}{2t}(2\pi)^{-\frac32}\int_{\R^3} e^{-i\frac{yz}{2t}}e^{-i\xi (x-y-z)}f_1(t,x-y)\ol{f_2(t,x)} f_3(t,x-z)\,dx dy dz\\
&=\frac{i\alpha}{2t}(2\pi)^{-\frac32}\int_{\R^3} \Big( e^{-i\frac{yz}{2t}}-1\Big) e^{-i\xi (x-y-z)}f_1(t,x-y)\ol{f_2(t,x)} f_3(t,x-z)\,dx dy dz\\
&\qquad +\frac{i\alpha}{2t}\xi^2 \left|\wh{f}(t,\xi)\right|^2\wh{f}(t,\xi ).
\end{align*}
The second term on the right-hand side is equal to $\wt{R}_1(t,\xi )-\xi R_1(t,\xi )$, while the first term can be estimated in the same way as for \eqref{eq:asymptotic0}, by
\[ t^{-1-\zeta}\Big( \| |x|^{2\zeta}\partial_xf(t)\|_{L^1}\| \partial_xf(t)\|_{L^1}\| f(t)\|_{L^1}+(\text{symm.~terms})\Big) \les t^{-1-\zeta +6\de}\ve_1^3.\]
Collecting the estimates above, we obtain, by setting $0<\delta \leq \frac{1}{12}$ and $\zeta= 7\de$,
\[ \| \xi R_1(t,\xi )\|_{L^\infty_\xi}\les t^{-1-\de}\ve_1^3. \]

Concerning $R_2$, we consider 
\begin{align*}
\wt{R}_2(t,\xi )&:=\xi \Big( R_2(t,\xi ) + \frac{i\beta}{2t}\xi \mathcal{H}\big[ |\wh{f}(t)|^2\big] (\xi )\wh{f}(t,\xi ) \Big) \\
&\;=\frac{\beta}{2\pi}\iint_{\R^2} e^{2it\eta \sigma}\xi^{2}\sgn{\eta}\wh{f}(t,\xi-\eta)\ol{\wh{f}(\xi-\eta-\sigma)}\wh{f}(\xi-\sigma )\,d\eta d\sigma 
\end{align*}
and decompose $\wt R_2$ into $\wt{R}_{2,1}+\wt{R}_{2,2}+\wt{R}_{2,3}$, similarly to $\wt R_1$ with \eqref{eq:algebraic}. The estimate of $\wt{R}_{2,1}$ is the same as that of $\wt{R}_{1,1}$, since the sign function $\sgn{\eta}$ does not affect the integration by parts in $\sigma$.
Indeed, we obtain, with some $0\leq k_j\leq 1$ satisfying $k_1+k_2+k_3=1$, 
\begin{align}\label{eq:esti-r21}
\begin{aligned}
	\| \wt{R}_{2,1}(t,\xi )\|_{L^\infty_\xi}&\les t^{-1} \left\| \partial_x^{k_1}u(t)\mathcal{H}\left[   {\rm Re}\left( \ol{e^{it\partial_x^2}x\partial_x^{k_2}f(t)}\partial_x^{k_3}u(t)\right) \right] \right\|_{L^1_x} \\
&\les t^{-1}\| u(t)\|_{H^1}\| f(t)\| _{H^{1,1}}\| u(t)\|_{W^{1,\infty}}  \\
&\les t^{-\frac32+5\de}\ve_1^3.
\end{aligned}
\end{align}
For $\wt{R}_{2,2}$, the situation is slightly complicated due to the sign function. By integration by parts in $\eta$, we estimate
\begin{align*}
\wt{R}_{2,2}(t,\xi ) &=\frac{\beta}{2\pi}\int_{\R^2} e^{2it\eta \sigma}\sigma \sgn{\eta}(\xi-\eta)\wh{f}(t,\xi-\eta)\ol{\wh{f}(\xi-\eta-\sigma)}\wh{f}(\xi-\sigma )\,d\eta d\sigma \\
&=-\frac{\beta}{2\pi}\frac{1}{2t}\int_{\R} \left\{ \Big( \int_0^\infty -\int_{-\infty}^0\Big) \partial_\eta (e^{2it\eta \sigma}) \wh{\partial_xf}(t,\xi-\eta)\ol{\wh{f}(\xi-\eta-\sigma)}\wh{f}(\xi-\sigma )\,d\eta \right\} d\sigma \\
&=\frac{\beta}{4\pi t}\int_{\R^2} e^{2it\eta \sigma}\sgn{\eta}\partial_\eta \Big( \wh{\partial_xf}(t,\xi-\eta)\ol{\wh{f}(\xi-\eta-\sigma)}\Big) \wh{f}(\xi-\sigma )\,d\eta d\sigma \\
&\quad +\frac{\beta}{2\pi t}\int_{\R} \wh{\partial_xf}(t,\xi)|\wh{f}(\xi-\sigma )|^2\,d\sigma .
\end{align*}
The first term of the right-hand side above is estimated similarly to $\wt{R}_{1,2}$, by
\[ t^{-1}\Big( \| e^{it\partial_x^2}x\partial_xf(t)\mathcal{H}\big[ |u(t)|^2\big] \|_{L^1}+\| \partial_xu(t)\mathcal{H}\big[ \ol{e^{it\partial_x^2}xf(t)}u(t)\big] \|_{L^1}\Big) \lesssim t^{-\frac32+5\de}\ve_1^3. \]
Hence, we obtain
\[ \wt{R}_{2,2}(t,\xi )= \frac{i\beta}{2\pi t}\xi \wh{f}(t,\xi) \int_{\R} |\wh{f}(t,\sigma )|^2\,d\sigma +O(t^{-1-\de} \ve_1^3 ). \]

To estimate $\wt{R}_{2,3}$, we repeat the approach for $\wt{R}_{1,3}$. Indeed, we have
\begin{align*}
\wt{R}_{2,3}(t,\xi )&= \frac{\beta}{2t}\mathcal{H}\big[ \ol{(-i\wh{\partial_xf})(t)}\wh{f}(t)\big] (\xi )(\wh{\partial_xf})(t,\xi)\\
&\quad +\frac{\beta}{2t}(2\pi)^{-\frac32}\int_{\R^3} \Big( e^{-i\frac{yz}{2t}}-1\Big) \sgn{z} e^{-i\xi (x-y-z)}\partial_xf(x-y)\ol{\partial_xf(x)} f(x-z)\,dx dy dz \\
&=  \frac{i\beta}{2t}\xi \mathcal{H}\big[ \cdot |\wh{f}(t,\cdot)|^2\big] (\xi ) \wh{f}(t,\xi) +O(t^{-1-\zeta+6\de}\ve_1^3) \\
&=   \frac{i\beta}{2t}\xi^{2} \mathcal{H}\big[ |\wh{f}(t,\cdot)|^2\big] (\xi) \wh{f}(t,\xi) - \frac{i\beta}{2\pi t}\xi \wh{f}(t,\xi)\int_{\R} |\wh{f}(t,\sigma )|^2\,d\sigma +O(t^{-1-\de}\ve_1^3),
\end{align*}
where  we used the commutator identity $[x,\mathcal{H}]f=\frac1{\pi}\int_{\R}f\,dx$.
Therefore, we see that
\[ \wt{R}_{2,2}(t,\xi )+\wt{R}_{2,3}(t,\xi )=\wt{R}_2(t,\xi )-\xi R_2(t,\xi )+O(t^{-1-\de} \ve_1^3). \]
Combining \eqref{eq:esti-r21} and estimates above, we finish the proof of \eqref{eq:asymptotic1}.

Let us consider \eqref{eq:linfty-local2}. For $0\leq t\leq 1$, the bound \eqref{eq:linfty-local2} follows from
\begin{align}\label{eq:interpol-weight}
\begin{aligned}
	\| \varrho_N(\xi )\bra{\xi}\wh{f}(t,\xi)\|_{L^\infty_\xi} &\les \| P_Nf(t)\|_{H^{1,\frac23}} \les \| P_Nf(t)\|_{H^1}^{\frac13}\| P_Nf(t)\|_{H^{1,1}}^{\frac23} \\
&\les N^{-\frac13}\| u(t)\|_{H^{2}}^{\frac13}\| f(t)\|_{H^{1,1}}^{\frac23} \les N^{-\frac13} \ve_1.
\end{aligned}
\end{align}
Since $B(t,\xi)$ is a real-valued function, \eqref{eq:time-deri-f} implies that
\[ \partial_t\left( \varrho_N(\xi)\bra{\xi}|\wh{f}(t,\xi )|\right)^2 =2{\rm Re}\left[ \bra{\xi} \left( R_1(t,\xi )+R_2(t,\xi)\right) \varrho_N(\xi)^2\bra{\xi}\ol{\wh{f}(t,\xi)}\right] . \]
For $t\geq 1$, the estimates \eqref{eq:interpol-weight} and \eqref{eq:linfty-apriori} yield that
\[ \| \varrho_N(\xi)\bra{\xi}\wh{f}(t,\xi)\|_{L^\infty_\xi}\lesssim \left( N^{-\frac13}t^{\frac{5\de}{3}}\right) ^\theta \ve_1 \]
for any $\theta \in [0,1]$.
Using this and \eqref{eq:asymptotic0}--\eqref{eq:asymptotic1}, we have
\[ \partial_t\left( \varrho_N(\xi)\bra{\xi}|\wh{f}(t,\xi )|\right)^2 \les t^{-1-\de}\ve_1^3\cdot N^{-\frac{\theta}{3}}t^{\frac{5\delta}{3}\theta}\ve_1.\]
Then this completes the proof by taking $\theta>0$ sufficiently small and integrating it over $[1,t]$.
\end{proof}


\section{Proof of the main theorem}\label{sec:main}

In this section, we prove the main theorem of this paper. 
For solutions $u\in C([0,T];H^2\cap H^{1,1})$ to \eqref{kdnls}, we use the $X_T$ norm introduced in \eqref{eq:assumption-apriori}.

We begin with the existence of global solutions $u\in C([0,\infty),H^2\cap H^{1,1})$ to \eqref{kdnls} for small initial data $\phi\in H^2\cap H^{1,1}$ satisfying \eqref{thm:initial}.
By the local well-posedness result in the previous work \cite{KL-LWP}, we have a unique local solution $u\in C([0,T];H^2\cap H^{1,1})$ with $\| u\|_{X_T}\les \ve$ for some $T>0$.
To extend it globally, it suffices to show a global a priori bound on $\| u(t)\|_{H^2\cap H^{1,1}}$.
Since 
\[ \| u(t)\|_{H^2\cap H^{1,1}}\les \| u(t)\|_{H^2}+\| xu(t)\|_{H^1}\les \bra{t}\| u(t)\|_{H^2}+\| J(t)u(t)\|_{H^1}\les \bra{t}^{1+\de}\| u\|_{X_T} \]
for $0\leq t\leq T$, it suffices to give an a priori bound on the $X_T$ norm. Indeed, we shall prove that
\begin{equation}\label{assumption:global}
\| u\|_{X_T}\leq \ve_1:=K\ve
\end{equation}
for any $T>0$ whenever $0 < \ve \ll 1$, where $K\geq 1$ is a large constant independent of $T$ and $\ve$.
The local result shows \eqref{assumption:global} for small $T$.
Suppose \eqref{assumption:global} holds for some $T>0$, then the a priori estimates 
\begin{align*}
	\|u(t)\|_{H^2} &\leq C(\ve + \bra{t}^\de \ve_1^2),\\
	\|xf(t)\|_{H^1} &\leq C(\ve + \bra{t}^{2\de} \ve_1^2)
\end{align*}
established in Sections \ref{sec:energy} and \ref{sec:weighted} imply that
\begin{align*}
	\|u\|_{X_T} \leq 2C(1+K^2\ve)\ve,
\end{align*} 
as long as $\ve_1=K\ve\leq 1$.
Choosing $K\geq 1$ large enough and $\ve_0>0$ small enough so that $2C(1+K^2\ve_0)\leq K/2$ and $K\ve_0\leq 1$, we have
\[ \| u\|_{X_T}\leq \ve_1/2 .\]
Hence, by a bootstrap argument we obtain \eqref{assumption:global}.
%


Next, we focus on proving the asymptotic behavior of the global solution.  
By Lemma~\ref{lem:asymptotic} (with $\ve_1=K\ve$), for $1\leq t_1 \leq t_2<\infty$ we have
\begin{align*}
	\left\|\bra{\xi}\left(e^{-iB(t_2,\xi)}\wh{f}(t_2,\xi) - e^{-iB(t_1,\xi)}\wh{f}(t_1,\xi) \right) \right \|_{L_\xi^\infty}  \les t_1^{-\de}\ve^3,
\end{align*}
where $B(t,\xi)$ is the real valued function given in the lemma.
Hence, 
the limit 
\begin{align*}
 W(\cdot ) = \lim_{t \to \infty} e^{-iB(t,\cdot )} \wh{f}(t,\cdot )
\end{align*}
exists in the sense of $\bra{\xi}^{-1}L^\infty_\xi$, and it holds that
\begin{align*}
	\normo{ \bra{\xi}  \left( \wh u(t,\xi) - e^{iB(t,\xi)}e^{it\xi^2} W(\xi) \right)}_{L_\xi^\infty} \les \bra{t}^{-\de} \ve^3
\end{align*}
for $t>0$. 
From the definition of $B(t,\xi)$ and calculations in Section \ref{sec:asymptotic}, we see that
\begin{align*}
	B(t,\xi) = \left[\frac \al2 \xi|W(\xi)|^2  + \frac \beta 2 \xi \cH \left(|W|^2\right)(\xi) \right] \log t + \Phi(\xi) + \ve^3O( t^{-\de_1}),
\end{align*}
uniformly in $\xi \in \R$, for some real-valued function $\Phi(\xi)$ and some $\de_1 \in (0,\de)$. This yields that
\begin{align*}
	\wh{f}(t,\xi)&= W(\xi)  \exp\left(i\left[\frac \al2 \xi|W(\xi)|^2  + \frac \beta 2 \xi \cH \left(|W|^2\right)(\xi)  \right] \log t + i\Phi(\xi) \right) + \ve^3 O( t^{-\de_1}).
\end{align*}
Then, by Lemma \ref{lem:linear} we obtain
\begin{align*}
	u(t)&= \frac1{\sqrt{2it}} e^{\frac{ix^2}{4t}} W\left(\frac x{2t}\right) \exp\left(i\left[ \frac {\al x}{4t}\left|W\left(\frac x{2t}\right) \right|^2  +  \frac {\beta x}{4t} \cH \left(\left|W\right|^2 \right) \left(\frac x{2t}\right) \right] \log t + i\Phi\left(\frac x{2t}\right)\right) \\
	&\qquad  +\ve^3 O(t^{-\frac12-\de_1})  +\ve \bra{t}^{2\de} O(t^{-\frac12-\gamma}),
\end{align*}
for any $\gamma \in \left(0,\frac14\right)$. This finishes the proof of \eqref{thm:asymptotic}.

Finally, we show the properties of  $\| u(t)\|_{L^2}$ in \eqref{thm:dissipation}. 
By Lemma~\ref{lem:L2energy}, the limit $D_\infty=\lim\limits_{t\to \infty}\| u(t)\|_{L^2}$ exists and $D_\infty = (1+ O(\ve^2)) \| \phi\|_{L^2}$. 
In particular, $D_\infty >0$ in the dissipative case $\beta <0$.
Moreover, we recall the following differential inequalities established in the proof of Lemma~\ref{lem:L2energy}:
\begin{align*}
	0 \le \sgn{\beta} \frac{d}{dt}\|u(t)\|_{L^2} \les \bra{t}^{-2+2\de} \ve^3
\end{align*}
for $t>0$.
Integrating it from $\infty$ to $t$, we obtain
\begin{align*}
	0 \le \sgn{\beta} \left(D_\infty - \|u(t)\|_{L^2}\right) \les \bra{t}^{-1+2\de} \ve^3
	\end{align*}
for $t\geq 0$.
This completes the proof of Theorem~\ref{thm:main}.
\hfill \qed 

\section*{Acknowledgement}
The authors would like to thank the anonymous referees for their careful reading of the manuscript and for their many valuable comments and suggestions, which were very helpful in improving the clarity and correctness of the paper.

N. Kishimoto was supported in part by KAKENHI Grants-in-Aid for Scientific Research (C) JP20K03678 and JP25K07066 funded by the Japan Society for the Promotion of Science (JSPS). K. Lee was supported
in part by RS-2025-00514043 and RS-2024-00463260, the National Research Foundation of
Korea(NRF) grant funded by the Korea government (MSIT) and (MOE), respectively.
\\

\textbf{Data Availability Statement.}  No data were generated or analyzed in this study.

\bibliographystyle{siam}
\bibliography{ReferencesNobuKiyeon}

\medskip

\end{document}